\documentclass[
final
]{dmtcs-episciences}

\usepackage[utf8]{inputenc}
\usepackage{subfigure}

\usepackage[round]{natbib}

\usepackage{amsmath}

\newtheorem{thm}{Theorem}
\newtheorem{cor}[thm]{Corollary}
\newtheorem{defi}[thm]{Definition}
\newtheorem{lem}[thm]{Lemma}

\newtheorem{claim}[thm]{Claim}

\newtheorem{obs}[thm]{Observation}

\newtheorem{remark}[thm]{Remark}

\usepackage{indentfirst}
\def\HH{\mathcal {H}}
\def\cH{\mathcal {H}}
\def\RR{\mathcal {R}}
\def\CC{\mathcal {C}}
\def\FF{\mathcal {F}}
\def\EE{\mathcal {E}}
\def\AA{\mathcal {A}}
\def\TT{\mathcal {T}}

\def\MM{\mathcal {M}}
\def\BB{\mathcal {B}}

\def\red1{\color{red}1\color{black}}
\def\blue1{\color{blue}1\color{black}}

\author[Bal\'azs Keszegh]{Bal\'azs Keszegh\affiliationmark{1,2}\thanks{Research supported by the J\'anos Bolyai Research Scholarship of the Hungarian Academy of Sciences, by the National Research, Development and Innovation Office -- NKFIH under the grant K 132696 and FK 132060, by the \'UNKP-21-5 and \'UNKP-22-5 New National Excellence Program of the Ministry for Innovation and Technology from the source of the National Research, Development and Innovation Fund and by the ERC Advanced Grant ``ERMiD''. This research has been implemented with the support provided by the Ministry of Innovation and Technology of Hungary from the National Research, Development and Innovation Fund, financed under the  ELTE TKP 2021-NKTA-62 funding scheme.}}
\title[A new discrete theory of pseudoconvexity]{A new discrete theory of pseudoconvexity}
\affiliation{
	Alfr\'ed R\'enyi Institute of Mathematics, Budapest, Hungary\\
	ELTE E\"otv\"os Lor\'and University, Budapest, Hungary}

\begin{document}
	
\publicationdata{vol. 25:1}{2023}{15}{10.46298/dmtcs.9255}{2022-03-28; 2022-03-28; 2022-09-30; 2023-04-25}{2023-04-28}

\maketitle

\begin{abstract}
	Recently geometric hypergraphs that can be defined by intersections of pseudohalfplanes with a finite point set were defined in a purely combinatorial way. This led to extensions of earlier results about points and halfplanes to pseudohalfplanes, including polychromatic colorings and discrete Helly-type theorems about pseudohalfplanes. 
	
	Here we continue this line of research and introduce the notion of convex sets of such pseudohalfplane hypergraphs. In this context we prove several results corresponding to classical results about convexity, namely Helly's Theorem, Carath\'eodory's Theorem, Kirchberger's Theorem, Separation Theorem, Radon's Theorem and the Cup-Cap Theorem. These results imply the respective results about pseudoconvex sets in the plane defined using pseudohalfplanes. 
	
	It turns out that most of our results can be also proved using oriented matroids and topological affine planes (TAPs) but our approach is different from both of them.  Compared to oriented matroids, our theory is based on a linear ordering of the vertex set which makes our definitions and proofs quite different and perhaps more elementary. Compared to TAPs, which are continuous objects, our proofs are purely combinatorial and again quite different in flavor. Altogether, we believe that our new approach can extend our understanding of these fundamental convexity results.
\end{abstract}

\section{Introduction}\label{sec:intro}

Given a (finite) point set $P$ and a family of regions $\RR$ (e.g., the family of all halfplanes) in the plane (or in higher dimensions), let $\HH$ be the hypergraph with vertex set $P$ and for each region of $\RR$ having a hyperedge containing exactly the same points of $P$ as this region. There are many interesting problems that can be phrased as a problem about hypergraphs defined this way, which are usually referred to as \emph{geometric hypergraphs}. This topic has a wide literature, researchers considered problems where $\RR$ is a family of halfplanes, axis-parallel rectangles, translates or homothets of disks, squares, convex polygons, pseudo-disks and so on. There are many results and open problems about the maximum number of hyperedges of such a hypergraph, coloring questions and other properties. For a survey of some of the most recent results see the introduction of \cite{abab} and of \cite{dp}, for an up-to-date database of such results with references see the webpage \cite{cogezoo}.\footnote{As this paper is in many ways a continuation of \cite{kbdiscretehelly} by the same author, the first three paragraphs of the introduction rely heavily on its introduction.}

One of the most basic families is the family of halfplanes, for which already many problems are non-trivial. Among others one such problem was considered in \cite{smorodinsky-yuditsky} where they prove that the vertices of every hypergraph defined by halfplanes on a set of points can be $k$-colored such that every hyperedge of size at least $2k+1$ contains all colors. In \cite{abafree} a generalization of this result was considered which we get by replacing halfplanes with the family of translates of an unbounded convex region (e.g., an upwards parabola). It turned out that the statement is true even when halfplanes are replaced by pseudohalfplanes. The main tools for proving these were the abstract notions of \emph{ABA-free hypergraphs} and \emph{pseudohalfplane hypergraphs}, shown to be equivalent to those hypergraphs that can be defined on points by upwards pseudohalfplanes and pseudohalfplanes, respectively.\footnote{The exact definition can be found later in Section \ref{sec:basicdef} and its connection to the geometric setting is deferred to Section \ref{sec:geom}.} This formulation had the promise that many other statements about halfplane hypergraphs can be generalized to pseudohalfplane hypergraphs in the future. While this combinatorial formulation has the disadvantage of being less visual and thus somehow less intuitive than the geometric setting, it has many advantages, among others covering a much wider range of hypergraphs, also, being purely combinatorial, it might have algorithmic applications as well. One recent application is a similar polychromatic coloring result about disks all containing the origin \cite{dp} where after observing that in every quadrant of the plane the disks form a family of pseudohalfplanes they can apply the results from \cite{abafree}.

In \cite{abafree} the equivalent of the convex hull vertices in the plane (more precisely, the points on the boundary of the convex hull) was defined for pseudohalfplane hypergraphs and called \emph{unskippable vertices} and this made it possible to generalize the proof idea of \cite{smorodinsky-yuditsky} from halfplanes to pseudohalfplane hypergraphs. To make it more intuitive, we call unskippable vertices \emph{extremal vertices} from here on. Exact definitions of these notions are postponed to Section \ref{sec:basicdef}.

We define convex sets of a pseudohalfplane hypergraph $\HH$ as sets that are intersections of some hyperedges of $\HH$ and we refer to these sets as \emph{pseudoconvex sets}. Notice that this is again in parallel with the geometric definition of convex sets (or more precisely, of the subsets of a base point set $P$ that we get by intersecting $P$ with convex sets). We have seen that already with halfplanes one can phrase many interesting problems, but using the notion of convex sets we can finally phrase many of the formative problems of discrete geometry, like the classical Helly's Theorem, Carath\'eodory's Theorem, Radon's Theorem, Erd\H os-Szekeres problem and the list goes on. While all of these problems are about discrete point sets, there are two essentially different types among them, in one type the whole statement is about some fixed point set $P$ while in the other type there is a point outside of $P$ that plays a role. E.g., in Carath\'eodory's theorem the whole statement is about a fixed point set,  while the classical Helly's theorem guarantees the existence of a new point in the plane with some property. Problems of the first type translate immediately to statements about pseudoconvex sets and it is interesting to see if they remain true in this more general setting. For the second type we can also pose a corresponding problem about pseudoconvex sets, where we want to extend the vertex set of the hypergraph $\HH$ with one or more vertices (we can extend the original hyperedges on the new vertices as we like) so that it remains a pseudohalfplane hypergraph and has the required property. Observe that for the first type a result about pseudoconvex sets implies the corresponding geometric result but for the second type such an implication does not immediately follow, although it still follows with a bit of additional work, as we will see later (illustrated in the proof of Theorem \ref{thm:weakhellypsconvexplane}).

Following this approach, we prove results about pseudoconvex sets that correspond to the planar case of some of the most important results of discrete geometry, namely Helly's Theorem, Carath\'eodory's Theorem, Radon's Theorem and the Cup-Cap Theorem. 

Finally, we discuss the relation of our definitions and results to previous similar results. A careful analysis reveals that we have mostly rediscovered things (namely Helly's Theorem, Carath\'eodory's Theorem, Radon's Theorem) that were known for a long time about oriented matroids (in particular about rank $3$ acyclic oriented matroids) or not so long about topological affine planes (TAPs, in short, for which the Cup-Cap theorem is also known). In addition to our direct proofs, we describe how the respective results about oriented matroids and TAPs imply our results and to what extent can these implications be reversed.

As said in \cite{bjorner}, in the past several people rediscovered what amounts to an axiom system for oriented matroids (or some special case thereof), without realizing that their work overlapped with already published papers. Our contribution can be regarded as an extension of this sequence of axiom systems by a new and interesting axiomatization of acyclic oriented matroids of rank $3$. However, we think that our methods are interesting on their own as they give a completely different approach based on hypergraphs on vertices that have a linear ordering on them. Also, while at the end our particular results are not stronger, formally our approach handles a bigger family of hypergraphs compared to what comes from rank $3$ oriented matroids.
Overall, the following sentence quoted from I.M. Gel'fand in \cite{bjorner} in relation to rediscoveries of the above mentioned axiom systems certainly applies to our case as well: "If you are not too ambitious, it can be a pleasure to realize that you have rediscovered something previously known, because at least then you know that you were on the right track."

\smallskip
The paper is structured as follows. After introducing the basic notions and previous Helly-type results about pseudohalfplanes in Section \ref{sec:basicdef} and Section \ref{sec:hellypshp}, in Section \ref{sec:convex} we generalize these Helly-type theorems to pseudoconvex sets, which are proved in Section \ref{sec:proofconvex}. The main tools used in our proofs are introduced in Section \ref{sec:extremal}. In the remainder of Section \ref{sec:intro} we show generalizations of Carath\'eodory's Theorem, Kirchberger's Theorem, Separation Theorem, Radon's Theorem and the Cup-Cap Theorem, these are proved in Section \ref{sec:proofclassical} and Section \ref{sec:proofcupcap}. In Section \ref{sec:geomconseq} we discuss the geometric consequences of our otherwise combinatorial results, the proofs of these connections can be found in Section \ref{sec:geom}. In Section \ref{sec:discussion} we give directions of possible future research. In Section \ref{sec:constr} we give simple examples showing the optimality of our Helly-type results. Finally, in Section \ref{sec:relations} we give a summary of related work and its connection to our results. In particular we discuss in detail to what extent and how the convexity results about TAPs and oriented matroids of rank $3$ imply many of our results and vice versa. While Section \ref{sec:relations} is quite long, we think it is useful to give a better understanding of the relation between these competing notions as most people tend to know one better than the others and so even for people familiar with them (let alone for the rest of us) it might not be immediate to see these connections.

\section{Basic definitions}\label{sec:basicdef}

\begin{defi}
	Given a hypergraph $\cH$ on vertex set $S$ and a subset $S'$ of $S$, the {\em subhypergraph of $\cH$ induced by $S'$} is the hypergraph on vertex set $S'$ with hyperedge set $\{ H \cap S': H \in \cH \}$ and it is denoted by $\cH[S']$.
\end{defi}

The above notion of induced subhypergaphs plays an important note in our studies. It can exhibit interesting behaviours, e.g., even if a hypergraph is uniform, its induced subhypergraph may not be so.

In the forthcoming sections we sometimes mention the geometric meaning of our definitions in the context of pseudohalfplanes in the plane. The exact definitions of pseudohalfplanes and hypergraphs defined by them can be found in Section \ref{sec:geomconseq}, thinking of them as normal halfplanes is usually good enough to get the right intuition.

As introduced in \cite{abafree}, ABA-free hypergraphs and their unskippable vertices are defined as follows.

\begin{defi}\label{def:ABA}
	A hypergraph $\mathcal F$ on an ordered vertex set is called {\em ABA-free} if $\mathcal F$ does not contain two hyperedges $A$ and $B$ for which there are three vertices $x<y<z$ such that $x,z\in A\setminus B$ and $y\in B\setminus A$.
\end{defi}

We imagine the vertices on a horizontal line, and thus if $x<y$ then we may say that $x$ is \emph{to the left from} $y$ and so on.

\begin{defi}
	In a hypergraph $\FF$ on an ordered vertex set, a vertex $a$ is {\em skippable} if there exists an $A\in \FF$ such that $\min(A)< a < \max(A)$ and $a\notin A$.
	In this case we say that $A$ {\em skips} $a$. 
	A vertex $a$ is {\em unskippable} if there is no such $A$.
\end{defi}

\begin{lem}\label{lem:unskippable}\cite{abafree}
	If $\mathcal F$ is ABA-free, then every $A\in \mathcal F$ contains an unskippable vertex.
\end{lem} 

It was shown in \cite{abafree} that ABA-free hypergraphs are exactly the geometric hypergraphs defined on planar points sets by upwards pseudohalfplanes. The intuition behind defining unskippable vertices is that if the ABA-free hypergraph is defined on a point set such that there is a hyperedge corresponding to every upwards halfplane, then the unskippable vertices are exactly those points that are on the geometric upper convex hull of the point set.

Let $\bar{\FF}$ denote the family of the complements of the hyperedges of $\FF$. It is easy to see and was shown in \cite{abafree} that if $\FF$ is ABA-free then $\bar{\FF}$ is also ABA-free. 

Now we are ready for the definition of pseudohalfplane hypergraphs as introduced in \cite{abafree}.

\begin{defi}
	A hypergraph $\HH$ on an ordered vertex set is a \emph{pseudohalfplane hypergraph} if there exists an ABA-free $\FF$ on the same ordered vertex set such that $\HH\subseteq \FF\cup \bar{\FF}$. Call $\TT=\HH\cap \FF$ the topsets and $\BB=\HH\cap \bar{\FF}$ the bottomsets, observe that both $\TT$ and $\BB$ are ABA-free. The unskippable vertices of $\FF$ (resp.\ $\bar{\FF}$) are called \emph{topvertices} (resp.\ \emph{bottomvertices}).
\end{defi}

Notice that the top- and bottomvertices depend only on $\FF$ and not on $\HH$ itself. For a given $\HH=\TT\cup\BB$ multiple $\FF$'s can witness that it is a pseudohalfplane hypergraph which can lead to different sets of top and bottomvertices. The smallest valid family is $\TT\cup\bar{\BB}$, which in particular gives the largest set of top and bottomvertices. For these reasons, when given a pseudohalfplane hypergraph $\HH$, even when not said explicitly, there is an ABA-free $\FF$ corresponding to it.

We call a pseudohalfplane hypergraph $\HH$ \emph{maximal} if no further hyperedge can be added without ruining that it is a pseudohalfplane hypergraph (with the given order of vertices).

It was shown in \cite{abafree} that pseudohalfplane hypergraphs are exactly the geometric hypergraphs defined on points sets by pseudohalfplanes.

Now we can proceed by defining the extremal vertices of a pseudohalfplane hypergraph, which generalize points on the convex hull of a set of points:

\begin{defi}	
	Given a pseudohalfplane hypergraph $\HH$ with corresponding ABA-free hypergaph $\FF$, the union of the topvertices and bottomvertices is called the \emph{extremal vertices} of $\HH$ and is denoted by $E(\HH,\FF)$ (or simply $E(\HH)$ when $\FF$ is clear from the context or even $E$ when $\HH$ is also clear from the context).\footnote{Notice that $E(\HH,\FF)$ actually depends only on $\FF$ and not on $\HH$, but it is more intuitive to think about it as the convex hull of $\HH$ nevertheless.}
\end{defi}

In light of the geometric setting the following is a natural way to define convex sets which turns out to be also very fruitful:

\begin{defi}\label{def:pseudoconvex}
	Given a hypergraph $\HH$ on an ordered set $S$ of vertices, the family of those subsets which are intersections of hyperedges of $\HH$ are called the \emph{convex sets} of $\HH$. The \emph{convex hull} of a subset $S'\subseteq S$ of the vertices is the convex set $Conv(S')=\cap \{H:H\in \HH,S'\subseteq H\}$ (where we define $\cap \emptyset:=S$).
\end{defi}

Clearly, given a point set $P$ in the plane, the subsets of $P$ which are defined by (geometric) convex sets are convex sets of the respective (pseudo)halfplane hypergraph. 

To state many of our results, we need the slightly technical definition of an extension of a pseudohalfplane hypergraph by new hyperedges or vertices:

\begin{defi}\label{def:ext}
	We say that a hypergraph $\HH'$ is an \emph{extension} of the pseudohalfplane hypergraph $\HH$ on vertex set $S$ if $\HH'$ is also a pseudohalfplane hypergraph on vertex set $S$ and $\HH$ is a proper subfamily of $\HH'$.
	
	We say that a pseudohalfplane hypergraph $\HH$ on vertex set $S$ \emph{can be extended by a vertex set $V$} (when $V=\{v\}$ we also say that it can be extended by the vertex $v$) if there exists a pseudohalfplane hypergraph $\HH^+$ on vertex set $S\cup V$ such that the topsets (resp. bottomsets) of $\HH^+$ are in bijection with the topsets (resp. bottomsets) of $\HH$ and for every such pair $H^+\in \HH^+$ and $H\in \HH$ we have $H^+\cap S=H$. Furthermore, if in $\HH$ two subsets $F,\bar F\subseteq S$ are complement-pairs on $S$ (i.e., one of them is a topset and the other is a bottomset, $F\cup \bar F=S$ and they are disjoint) then the corresponding hyperedges $F^+$ and $\bar F^+$ of $\HH^+$ are complement-pairs on $S\cup V$.\footnote{Thus in a sense when extending $\HH$ we in fact extend the corresponding ABA-free $\FF$.} We also say that $\HH^+$ is an extension of $\HH$ to the vertex set $V$ (or to the vertex $v$ if $V=\{v\}$). The extension of a convex set $C= H_1\cap\dots \cap H_l$ of $\HH$ is the convex set $C^+=H_1^+\cap\dots \cap H_l^+$ of $\HH^+$.
\end{defi}

The \emph{dual} of a hypergraph $\HH$ is the hypergraph we get if we reverse the role of vertices and hyperedges while reversing the incidence relation. 

Duals of pseudohalfplane hypergraphs were studied in \cite{abafree} alongside pseudohalfplane hypergraphs.
In addition, pseudohemisphere hypergraphs were defined and studied, a common generalization of pseudohalfplane hypergraphs and duals of pseudohalfplane hypergraphs.

\begin{defi}\label{def:pshemi}\cite{abafree}
	A {\em pseudohemisphere hypergraph} is a hypergraph $\HH$ on an ordered set of vertices $S$ such that there exists a set $X\subset S$ and an ABA-free hypergraph $\cal F$ on $S$ such that the hyperedges of $\HH$ form some subset of $\{F\Delta X, \bar F \Delta X \mid F\in \cal F\}$ (where $A\Delta B$ denotes the set $(A\setminus B)\cup (B\setminus A)$).
\end{defi}

In \cite{abafree} it was shown that given a point set on a sphere in $3$ dimensions and an arrangement of geometric \emph{pseudohemispheres}, i.e., regions whose boundaries are centrally symmetric simple curves such that any two intersect exactly twice, they define a pseudohemisphere hypergraph ($X$ is the set of vertices contained in one of the pseudohemispheres).

We present some convexity results about pseudohemisphere hypergraphs as well. Note that Definition \ref{def:ext} about extensions can be easily modified for pseudohemisphere hypergraphs.

\section{Results}

\subsection{Helly theorems for pseudohalfplanes}\label{sec:hellypshp}

In \cite{abafree} already some discrete Helly-type theorems were proved about pseudohalfplane hypergraphs and also about pseudohemisphere hypergraphs:

\begin{lem}[Primal Discrete Helly theorem for pseudohalfplanes, $3\rightarrow +1$]\cite{abafree}\label{lem:weakhellypshp}
	Given a pseudohalfplane hypergraph $\HH$ such that every triple of hyperedges has a common vertex, then we can extend $\HH$ to a pseudohalfplane hypergraph by a vertex contained in every hyperedge of the extension.
\end{lem}

\begin{lem}[Primal Discrete Helly theorem for pseudohemispheres,  $4\rightarrow +1$]\cite{abafree}\label{lem:weakhellypshs}
	Given a pseudohemisphere hypergraph $\HH$ such that every $4$-tuple of hyperedges has a common vertex, then we can extend $\HH$ to a pseudohemisphere hypergraph by a vertex contained in every hyperedge of the extension.
\end{lem}

In \cite{abafree} it is not said, so we note here that these two statements are best possible. In Section \ref{sec:constr} we give the examples that show this.

Considering if Lemma \ref{lem:weakhellypshp} has a dual, dual Helly theorems are meaninglessly true for pseudohalfplane hypergraphs as we can always add a new hyperedge containing all vertices and the hypergraph remains to be a pseudohalfplane hypergraph. For pseudohemisphere hypergraphs this is not true anymore. In this case, using that the dual of a pseudohemisphere hypergraph is also a pseudohemisphere hypergraph \cite{abafree}, Theorem \ref{thm:weakhellypshsconvex} (stated later) will imply the following:

\begin{thm}[Dual Discrete Helly theorem for pseudohemispheres, $4\rightarrow +1$]\label{thm:dualweakpshs}
	Given a pseudohemisphere hypergraph $\HH$ on vertex set $S$ and a subset $S'\subseteq S$ of its vertices such that every $4$-tuple of the vertices of $S'$ is contained in some hyperedge, then we can add a hyperedge containing all vertices of $S'$ (so that $\HH$ together with this new hyperedge is a pseudohemisphere hypergraph).
\end{thm}

Recently in \cite{jjr} they proved discrete Helly-type theorems which can be formulated in terms of halfplane hypergraphs, their results were extended by the author to pseudohalfplane hypergraphs \cite{kbdiscretehelly}. In these results while we cannot hit all hyperedges with one vertex, on the other hand we can choose the vertex from the original vertex set.\footnote{To distinguish from the rest of the discrete Helly theorems, we will use the word \emph{strong} to refer to the fact that the vertex found is in the original vertex set. This is somewhat similar to the difference between (strong) $\epsilon$-nets and weak $\epsilon$-nets} Instead of listing all these results, we mention just one which is closest to Lemma \ref{lem:weakhellypshp}:

\begin{thm}[Primal Strong Discrete Helly theorem for pseudohalfplanes, $3\rightarrow 2$]\label{thm:primalpshp32}\cite{kbdiscretehelly}
	Given a pseudohalfplane hypergraph $\HH$ such that every triple of hyperedges has a common vertex, there exists a set of at most $2$ vertices that hits every hyperedge of $\HH$.
\end{thm}

\subsection{Helly theorems for pseudoconvex sets}\label{sec:convex}

We are able to generalize Helly's Theorem for pseudoconvex sets:

\begin{thm}[Discrete Helly theorem for pseudoconvex sets, $3\rightarrow +1$]\label{thm:weakhellypsconvex}
	Given a pseudohalfplane hypergraph $\HH$ and a subfamily  $\CC$ of its convex sets such that every triple of convex sets from $\CC$ has a common vertex, then we can extend $\HH$ to a pseudohalfplane hypergraph by a vertex contained in every convex set which is an extension of a set from $\CC$.
\end{thm}

Similarly, we can show such a result for pseudohemispheres.

\begin{thm}[Discrete Helly theorem for convex sets of pseudohemispheres,  $4\rightarrow +1$]\label{thm:weakhellypshsconvex}
	Given a pseudohemisphere hypergraph $\HH$ and a subfamily $\CC$ of its convex sets such that every $4$-tuple of convex sets from $\CC$ has a common vertex, then we can extend $\HH$ to a pseudohemisphere hypergraph by a vertex contained in every convex set which is an extension of a set from $\CC$.
\end{thm}

These generalize the previously mentioned discrete Helly results Lemma \ref{lem:weakhellypshp} and Lemma \ref{lem:weakhellypshs} from \cite{abafree} in two ways. As a first step, instead of $\HH$ we can consider a subfamily of $\HH$ (similar to Theorem \ref{thm:dualweakpshs}, where we considered a subset $S'$ of $S$ instead of the whole $S$). As a second step, this family is actually not required to be a subfamily of $\HH$, instead it has to be only a subfamily $\CC$ of the convex sets of $\HH$. 

These two statements are best possible, in Section \ref{sec:constr} we give the examples that show this.

\subsection{Carath\'eodory's, Kirchberger's and the Separation theorem}

We can also generalize Carath\'eodory's Theorem but here we have to be more careful. The obvious way would be to claim that given a pseudohalfplane hypergraph $\HH$ on vertex set $S$ and a subset $S'\subseteq S$ of its vertices and a vertex $v\in S$, if $v\in Conv(S')$ then there exists an $S''\subseteq S'$, $|S''|\le 3$ such that already $v\in Conv(S'')$. For this statement however there is a simple counterexample. Let the hyperedges be $S$ and all the subsets of $S$ of size $|S|-2$ that do not contain $v$. It is easy to see that this is an ABA-free hypergraph (for an arbitrary ordering of the vertices). One can verify this by checking the definition of ABA-freeness. Alternatively one can easily realize this hypergraph by points in the plane and upwards halfplanes, see Figure \ref{fig:cara}, which implies that it is ABA-free (for the implication see Section \ref{sec:geomconseq}). On the other hand while for $S'=S\setminus\{v\}$ we have $v\in S=Conv(S')$, for any proper subset $S''$ of $S'$ we have that $v\notin S''=Conv(S'')$, thus a statement like we hoped for cannot be true even if we replace $3$ with some other constant.\footnote{Notice that if we also take the complements of the sets we get such an example which is a pseudohalfplane hypergraph $\HH'$ in which the convex hull of every subset of size two contains no third vertex, that is, the vertices are in \emph{general position} in some sense.}

\begin{figure}
	\begin{center}
		\includegraphics[height=4cm]{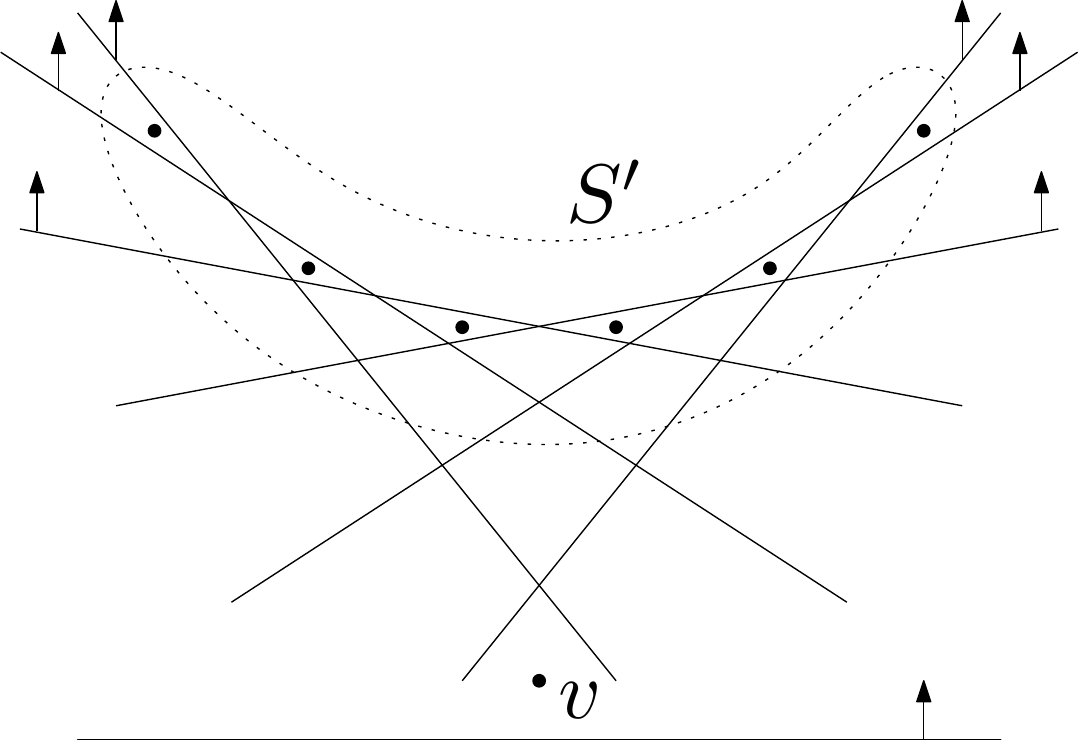}
		\caption{Realization with upwards halfplanes a pseudohalfplane hypergraph in which $v\in Conv(S')$ yet $v$ is not in the convex hull of any proper subset of the vertices of $S'$.}
		\label{fig:cara}
	\end{center}
\end{figure}

It turns out that the useful stronger way to define $v$ `being inside the convex hull' of $S'$ is the following:
\begin{defi}
	Given a pseudohalfplane hypergraph $\HH$ on vertex set $S$. If for some vertex $v$ and subset $S'$ of the vertices of $\HH$, in any extension $\HH'$ of $\HH$ with additional hyperedges it is still true that $v\in Conv(S')$ in $\HH'$ (i.e., when the convex hull is defined with respect to $\HH'$) then we say that \emph{$v$ is strongly inside the convex hull of $S'$}.
\end{defi}  
It is easy to see that when $\HH$ is defined by points and halfplanes in the plane then a point is strongly inside the convex hull of a set of points $S'$ if and only it is in the interior (defined as the largest open subset) of the convex hull of $S'$. Note that trivially if $v$ is strongly inside the convex hull of $S'$ then also $v\in Conv(S')$ in $\HH$. 
The following claim gives an intuitive equivalent definition for which there is no need to consider extensions of $\HH$:

\begin{claim}\label{claim:strongconvex}
	Given a pseudohalfplane hypergraph $\HH$ on vertex set $S$, $v\in S$ is strongly inside the convex hull of $S'\subseteq S$ if and only if $v$ is not an extremal vertex of $\HH[S'\cup \{v\}]$.
\end{claim}

With this definition we get a true statement analogous to and implying the planar case of the classical theorem of Steinitz which states that if $X$ is a set in $\RR^d$ and $p$ is in the interior of the convex hull of $X$, then $p$ is in the interior of the convex hull of some subset $Y$ of $X$ of size at most $2d$ (see, e.g., \cite{Eckhoff}):

\begin{claim}[Steinitz's theorem for pseudoconvex sets]\label{claim:pseudosteinitz} 
	Given a pseudohalfplane hypergraph $\HH$ on vertex set $S$ and a subset $S'\subseteq S$ of its vertices and a vertex $v\in S$. If $v$ is strongly inside the convex hull of $S'$ then there exists an $S''\subseteq S'$, $|S''|\le 4$ such that $v$ is strongly inside the convex hull of $S''$.
\end{claim}

However, the proof is nearly trivial from the definitions. Moreover we cannot replace $4$ with $3$ as shown by the halfplane hypergraph defined on the point set of four points $S'=S$ in convex position and a fifth point $v$ in the intersection of the diagonal of this $4$-gon. Indeed, here while $v$ is not an extremal vertex of $S'\cup \{v\}$, for any triple of points $S''$ from $S'$, $v$ is an extremal vertex of $S''\cup\{v\}$\footnote{Not surprisingly the same example shows that if in Carath\'eodory's Theorem we would want that $v$ is not only in the convex hull but also in the interior of the convex hull of $S''$ then we also may need $S''$ to have $4$ vertices.}.

Thus we prove the following combination of the two containments, which from its proof seems to be the right generalization of Carath\'eodory's theorem:

\begin{thm}[Carath\'eodory's theorem for pseudoconvex sets]\label{thm:pseudocaratheodory}
	Given a pseudohalfplane hypergraph $\HH$ on vertex set $S$ and a subset $S'\subseteq S$ of its vertices and a vertex $v\in S$. If $v$ is strongly inside the convex hull of $S'$ then there exists an $S''\subseteq S'$, $|S''|\le 3$ such that $v\in Conv(S'')$ in $\HH$. Moreover every vertex of $S''$ can be chosen to be extremal in $S'$.
\end{thm}

Note that this indeed directly implies the geometric Carath\'eodory's theorem as given a point set in the plane we can slightly push inside from the convex hull of $S'$ the points which are on the boundary of the convex hull of $S'$ but are not vertices of it (including possibly $v$). Then we can apply Theorem \ref{thm:pseudocaratheodory} on the halfplane hypergraph that this point set induces to get a triple $S''$ of extremal points of $S'$ such that $v\in Conv(S'')$. As these points were not moved in the previous step, and $v$ was moved only slightly, we can conclude that $v$ must be in the convex hull of $S''$ in the geometric sense as well (possibly on its boundary).

\begin{defi}
	Given three subsets $A,B,H$ of a vertex set $S$, with $A\cap B=\emptyset$, $H$ \emph{separates} $A$ and $B$ if $A\subseteq H$ and $H\cap B=\emptyset$ or if $H\cap A=\emptyset$ and $B\subseteq H$.
\end{defi}	

\begin{obs}
	Given a pseudohalfplane hypergraph $\cH$ on vertex set $S$ and three subsets $A,B,H$ of $S$. If $H$ separates $A$ and $B$ and both $H$ and $\bar H$ are hyperedges of $\cH$ then in $\cH$ $Conv(A)\cap Conv(B)=\emptyset$.
\end{obs}

The following are the Hahn-Banach Separation theorem and Kirchberger's theorem for pseudoconvex sets:

\begin{thm}[Separation theorem for pseudoconvex sets]\label{thm:sep}
	Given a pseudohalfplane hypergraph $\HH$ on vertex set $S$ and subsets $A,B\subset S$.
	If we cannot extend $\HH$ with a new vertex $v$ such that $v\in Conv(A)\cap Conv (B)$ in this extension then we can extend $\HH$ with a new hyperedge $H$ such that $H$ separates $A$ and $B$.\footnote{
		Note that there is an extension with a vertex $v$ such that $v\in Conv(A)\cap Conv (B)$ if and only if there is an extension with a vertex $v$ such that $Conv(A)\cap Conv (B)\ne \emptyset$.}
\end{thm}

The halfplane hypergraph induced by $4$ points in convex position shows that in Theorem \ref{thm:sep} it is not enough to assume that $Conv(A)\cap Conv (B)=\emptyset$. Indeed, let opposite pairs of vertices be the sets $A$ and $B$, then $Conv(A)\cap Conv (B)=\emptyset$ yet it can be checked (e.g, with the help of a computer program, for a similar argument see Claim \ref{claim:no2-1}) that we cannot separate $A$ and $B$ by a new hyperedge. The vertex $v$ in the intersection of the diagonals (forming the $5$ point construction from before) shows that the stronger assumption we used in Theorem \ref{thm:sep} does fail, see also Lemma \ref{lem:4vertices}.

\begin{thm}[Kirchberger's theorem for pseudoconvex sets]\label{thm:kirch}
	Given a pseudohalfplane hypergraph $\HH$ on vertex set $S$ and subsets $A,B\subset S$. If for every subset $D\subset S$ with $|D|\le 4$ there exists a hyperedge of $\HH$ separating $A\cap D $ and $B\cap D$, then there exists an extension $\HH'$ of $\HH$ with one new hyperedge $H$ such that $H$ separates $A$ and $B$.	
\end{thm}

We will prove the following theorem which immediately implies the previous two theorems:
\begin{thm}\label{thm:multiequi}
	Given a pseudohalfplane hypergraph $\HH$ on vertex set $S$ and disjoint subsets $A,B\subset S$, the following are equivalent:
	\begin{enumerate}
		\item There exists an extension $\HH'$ of $\HH$ with additional hyperedges such that $\HH'$ cannot be extended with a new vertex $v$ such that $v\in Conv(A)\cap Conv(B)$ in $\HH'$.
		\item There exists an extension $\HH'$ of $\HH$ with additional hyperedges such that $\HH'$ cannot be extended with a new vertex $v$ such that for some subset $D\subseteq S$ with $|D|\le 4$ we have $v\in Conv(A\cap D)\cap Conv(B\cap D)$ in $\HH'$.
		\item There exists an extension $\HH'$ of $\HH$ with one new hyperedge $H$ such that $H$ separates $A$ and $B$.
	\end{enumerate}
\end{thm}

We will use Theorem \ref{thm:pseudocaratheodory} in the proof of Theorem \ref{thm:multiequi}. Theorems \ref{thm:sep} and \ref{thm:kirch} will have a short proof using Theorem \ref{thm:pseudocaratheodory}. Note that similar to how the classic Kirchberger's Theorem implies the classic Carathéodory's Theorem, setting $A=\{v\}$ and $S'=S\setminus\{v\}$, Theorem \ref{thm:multiequi} implies Theorem \ref{thm:pseudocaratheodory}. 


%
%


\subsection{Radon's theorem}

Using our discrete Helly theorems we can prove a discrete Radon's theorem for pseudoconvex sets:

\begin{thm}[Radon's theorem for pseudoconvex sets]\label{thm:radon}
	Given a pseudohalfplane hypergraph $\HH$ on vertex set $S$ and a subset $S'\subseteq S$ with $|S'|=4$, then we can extend the hypergraph with a new vertex $v$ so that it is still a pseudohalfplane hypergraph and there is a partition of $S'$ into two subsets such that the convex hulls of these subsets both contain $v$.
\end{thm}

Note that we cannot guarantee that $v$ is already in $S$, as already in the plane if we have a point set $P$ in convex position then for any two disjoint subsets of the points their convex hull does not contain any of the points of $P$.

\subsection{Cup-Cap Theorem}

We aim to show a generalization of the well-known Cup-Cap Theorem by Erd\H os and Szekeres, to state it we first need to define cups and caps of pseudohalfplane hypergraphs.

\begin{defi}
	Given a pseudohalfplane hypergraph $\HH$ on vertex set $S$ and a subset $A$ of the vertices. We say that $A$ forms a \emph{cup} (resp. \emph{cap}) if in $\HH[A]$ every vertex is a bottomvertex (resp.\ topvertex). A cup (resp. cap) on $k$ vertices is called a \emph{$k$-cup} (resp.\ a \emph{$k$-cap}).
\end{defi}

It is easy to see that when defined on a halfplane hypergraph these coincide with the 
geometric notions of cups and caps.

\begin{thm}[Cup-Cap Theorem for pseudoconvex sets]\label{thm:pseudoesz}
	Given a pseudohalfplane hypergraph $\HH$ on vertex set $S$ of size $|S|\ge {\binom{k+l-4}{k-2}}+1$, it contains either a $k$-cup or an $l$-cap.
\end{thm}

Note that this theorem is different from our other results in the sense that it does not require the notion of the convex hull of a subset, instead it relies on the notion of extremal vertices.

\subsection{Geometric variants}\label{sec:geomconseq}

The following brief summary of the relevant geometric notions is borrowed from the preceding companion paper \cite{kbdiscretehelly}.

\smallskip
\textbf{Pseudolines.}
A {\em pseudoline arrangement} is a finite collection of simple curves (Jordan arcs) in the plane such that each curve cuts the plane into two components (i.e., both endpoints of each curve are at infinity) and any two of the curves are either disjoint or intersect once, and in the intersection point they cross. It is usually also required and so we require as well that they intersect exactly once. 
However, this restriction is not always needed and thus to be able to distinguish the two variants, we call a \emph{loose} pseudoline arrangement an arrangement where not all pairs of pseudolines intersect. An arrangement of pseudolines is \emph{simple} if no three pseudolines meet at a point. Wlog. we can assume that the pseudolines are $x$-monotone bi-infinite curves (see, e.g., \cite{abafree}), such arrangements are sometimes called \emph{Euclidean} or \emph{graphic} pseudoline arrangements. 

One can regard a pseudoline arrangement as a plane graph (with half-infinite edges): the vertices of an arrangement of pseudolines are the intersection points of the pseudolines, the edges are the maximal connected parts of the pseudolines that do not contain a vertex and the faces are the maximal connected parts of the plane which are disjoint from the edges and the vertices of the arrangement.\footnote{The vertices of an arrangement should not be confused with the vertices of a hypergraph.}
We say that two pseudoline arrangements are (combinatorially) {\em equivalent} if there is a one-to-one adjacency-preserving correspondence between their pseudolines, vertices, edges and faces.

For an introduction into pseudoline arrangements see Chapter 5 of \cite{handbook} by Felsner and Goodman.

\smallskip
\textbf{Pseudohalfplanes.}
Given a pseudoline arrangement, a \emph{pseudohalfplane family} is the subfamily of the above defined components (one on each side of each pseudoline). A pseudohalfplane family is simple (resp. loose) if the boundaries form a simple (resp. loose) pseudoline arrangement.
A pseudohalfplane family is upwards if we just take components that are above the respective pseudoline (in which case the pseudolines are assumed to be $x$-monotone).

In \cite{abafree} it is shown that given a family $\EE$ of pseudohalfplanes in the plane and a set of points $P$ then the hypergraph whose hyperedges are the subsets that we get by intersecting regions of $\EE$ with $P$ is a pseudohalfplane hypergraph, and that all pseudohalfplane hypergraphs can be realized this way.\footnote{In fact they only prove that we can realize them with loose simple pseudoline arrangements but their argument can be easily modified to have a realization with a simple and not loose pseudoline arrangement as well.} If $\EE$ is a family of upwards pseudohalfplanes then we get the ABA-free hypergraphs and all ABA-free hypergraphs can be realized with upwards pseudohalfplanes. Thus, all our results about pseudohalfplane hypergraphs implies the respective result about (loose and not loose) families of pseudohalfplanes where we replace vertices with points and hyperedges with pseudohalfplanes.

\bigskip
Our results are strictly combinatorial as they are about pseudohalfplane hypergraphs. Nevertheless, we claim that in these results we can replace pseudohalfplane hypergraphs with families of pseudohalfplanes and instead of adding a new vertex $v$ (in case the statement claims so) we can actually find a point in the plane with the required property. 

Given a family of pseudohalfplanes $\EE$ then for a (not necessarily finite) set of points $P$ we define $Conv(P)$ as the intersection of all pseudohalfplanes that contain $P$. A subset of the plane is convex (with respect to the given pseudohalfplane family) if $P=Conv(P)$, that is, if it is the intersection of the pseudohalfplanes of a subfamily of $\EE$.

While we omit translating all our results to this geometric setting, we provide the translation for three of them, which we prove in Section \ref{sec:geom}.

\begin{thm}[Strong Discrete Helly theorem for pseudoconvex sets in the plane, $3\rightarrow 1$]\label{thm:weakhellypsconvexplane}
	Given a finite family of pseudohalfplanes $\EE$. Let $\CC$ be a subfamily of the convex subsets of the plane with respect to $\EE$ such that every triple of convex sets from $\CC$ has a non-empty intersection, then there is a point in the plane which is in every member of $\CC$.
\end{thm}

To be able to state the next theorem we define $p$ being strongly inside the convex hull of $P$ in the geometric setting if in every extension of $\EE$ with an additional pseudohalfplane we still have that $p\in Conv(P)$ with respect to this extended family. Similarly, $p$ is an extremal vertex of $P$ if there exists an extension of $\EE$ by an additional pseudohalfplane such that $p\notin Conv(P\setminus \{p\})$ with respect to this extended family.

\begin{thm}[Carath\'eodory's theorem for pseudoconvex sets in the plane]\label{thm:pseudocaratheodoryplane}
	Given a finite family of pseudohalfplanes $\EE$. Let $P'$ be a set of points in the plane and $\CC$ be a subfamily of the convex subsets of the plane with respect to $\EE$. If $p$ is strongly inside the convex hull of $P'$ (with respect to $\EE$) then there exists a $P''\subseteq P'$, $|P''|\le 3$ such that $p\in Conv(P'')$ and every point of $P''$ is an extremal vertex (with respect to $\EE$).
\end{thm}

Note that there is no need in the above theorem to have a superset $P$ of $P'$ (like we had in Theorem \ref{thm:pseudocaratheodory}) as while in the combinatorial setting $P$ serves as the surrounding space, here the pseudohalfplanes are in the plane and so the plane itself serves as the surrounding space (as seen in the proof later in Section \ref{sec:geom}).
Note also that due to the construction on Figure \ref{fig:cara} it is not enough to assume that $p\in Conv(P)$. 

\begin{thm}[Radon's theorem for pseudoconvex sets in the plane]\label{thm:radonplane}
	Given a finite family of pseudohalfplanes $\EE$. Let $P'$ be a set of points in the plane with $|P'|=4$. There is a partition of $P'$ into two subsets such that the convex hulls of these subsets (with respect to $\EE$) intersect.
\end{thm}

Besides the geometric counterparts of our results, another interesting geometric consequence is a short proof of Levi's enlargement lemma, which follows easily from Lemma \ref{lem:discretelevi}, a discrete version of Levi's enlargement lemma, whose proof is also quite short.

\section{Properties of pseudohalfplane hypergraphs}\label{sec:extremal}

\subsection{Extensions of pseudohalfplane hypergraphs}

From now on whenever we define a pseudohalfplane hypergraph as the subfamily of $\FF\cup\bar{\FF}$ for some hypergraph $\FF$, then implicitly we assume that $\FF$ is ABA-free.

We show two very useful lemmas about extending pseudohalfplane hypergraphs:

\begin{lem}[Extension with a vertex]\label{lem:extension}
	Given a pseudohalfplane hypergraph $\HH$ and a subhypergraph $\HH'\subseteq\HH$ on vertex set $S$. Suppose that $\HH'^+$ is an extension of $\HH'$ to an additional vertex $v$. Then we can also extend $\HH$ to $v$ to get $\HH^+$ so that $\HH'^+$ is a subhypergraph of $\HH^+$.
\end{lem}

\begin{proof}
	Notice that it is enough to show this when $\HH$ and $\HH'^+$ are ABA-free hypergraphs and we want that $\HH^+$ is also ABA-free. Indeed, each pseudohalfplane hypergraph by definition has an underlying ABA-free hypergraph. Applying the statement to this ABA-free hypergraph we get the statement also for the original pseudohalfplane hypergraph.
	
	Thus, from now on we can suppose that $\HH$ is an ABA-free hypergraph. Denote by $
	\HH''=\HH\setminus \HH'$.
	We proceed by induction, 
	let $H\in \HH''$ be an arbitrary hyperedge in $\HH''$. We want to extend $H$ to $v$ to get a hyperedge $H^+$ on vertex set $S'=S\cup\{v\}$ such that $\HH'^+\cup\{H^+\}$ is a pseudohalfplane hypergraph. Clearly, repeating this step one by one for every hyperedge in $\HH''$ (always extending the hypergraph that we got in the previous step) we get the desired hypergraph $\HH^+$.
	
	There are two possible extensions of $H$, either $v$ is in $H^+$ or not in $H^+$. Assume on the contrary that both extensions ruin ABA-freeness. Thus, when not adding $v$ to $H$ we get an occurrence of ABA on $H^+$ and some other hyperedge $H_1\in \HH'$. The vertex $v$ must be one of the three vertices of this occurrence, otherwise $\HH'^+$ would not have been ABA-free. Similarly, adding $v$ to $H$ we also get an occurrence of ABA on $H^+$ and some other hyperedge $H_2\in \HH'$ s.t. $v$ must be one of the three vertices of this occurrence. 
	
	Overall, up to symmetry and possibly taking the complement of every hyperedge we have some cases depending on which vertex of the 3-3 vertices of the above two ABA-occurrences is $v$ in the vertex order (the middle one or one of the two side ones). See Figures \ref{fig:extension},\ref{fig:extension2},\ref{fig:extension3},\ref{fig:extension4} for the rest of the proof, in every figure the dots (corresponding to containment) and circles (corresponding to non-containment) are numbered according to the order they are deducted, leading to an ABA-occurrence (a contradiction) with hyperedges and vertices with blue labels. Except for the first case, the easy details are left to the reader. 
	
	\begin{enumerate}
		\item Middle + Side case. There are vertices $a<v<b$ s.t. $H\cap \{a,b\}=\{a,b\}$ and $H_1\cap \{a,v,b\}=\{v\}$. There are vertices $c<d<v$ s.t. $H\cap \{c,d\}=\{c\}$ and $H_2\cap \{c,d,v\}=\{d\}$.
		
		 Notice that $a\ne d$ as only one of them is in $H$.  Now $b\in H_2$ otherwise there would be an ABA-occurrence on $c,d,b$ and on $H,H_2$. Also $a\notin H_2$ otherwise there would be an ABA-occurrence on $a,v,b$ and $H_1,H_2$. Also $d\in H_1$ otherwise there would be an ABA-occurrence on $d,v,b$ and $H_1,H_2$. Now if $d<a$ then there is an ABA-occurrence on $c,d,a$ and $H,H_2$ while if $a<d$ then there is an ABA-occurrence on $a,d,b$ and $H,H_1$. Both cases lead to a contradiction.		
		 
		 \item Middle + Middle case. There are vertices $a<v<b$ s.t. $H\cap \{a,b\}=\{a,b\}$ and $H_1\cap \{a,v,b\}=\{v\}$. There are vertices $c<v<d$ s.t. $H\cap \{c,d\}=\emptyset$ and $H_2\cap \{c,d,v\}=\{c,d\}$.
		 
		 \item Side + Same Side case. There are vertices $a<b<v$ s.t. $H\cap \{a,b\}=\{b\}$ and $H_1\cap \{a,b,v\}=\{a,v\}$. There are vertices $c<d<v$ s.t. $H\cap \{c,d\}=\{c\}$ and $H_2\cap \{c,d,v\}=\{d\}$.
		 
		 \item Side + Opposite Side case. There are vertices $a<b<v$ s.t. $H\cap \{a,b\}=\{b\}$ and $H_1\cap \{a,b,v\}=\{a,v\}$. There are vertices $v<c<d$ s.t. $H\cap \{c,d\}=\{d\}$ and $H_2\cap \{v,c,d\}=\{c\}$.
	\end{enumerate}
\end{proof}

\begin{figure}
	\begin{center}
		\includegraphics{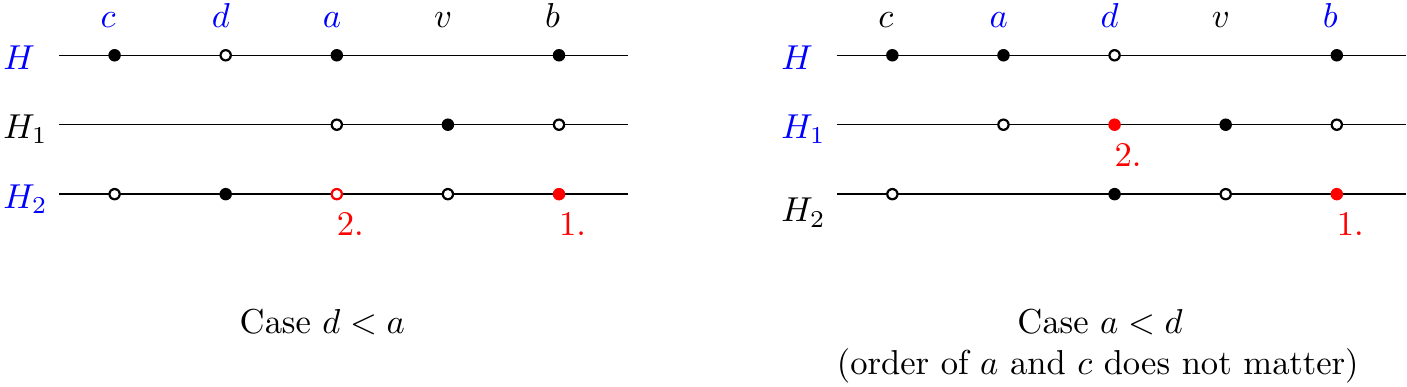}
		\caption{Lemma \ref{lem:extension}, Middle + Side case.}
		\label{fig:extension}
	\end{center}
\end{figure} 

\begin{figure}
	\begin{center}
		\includegraphics{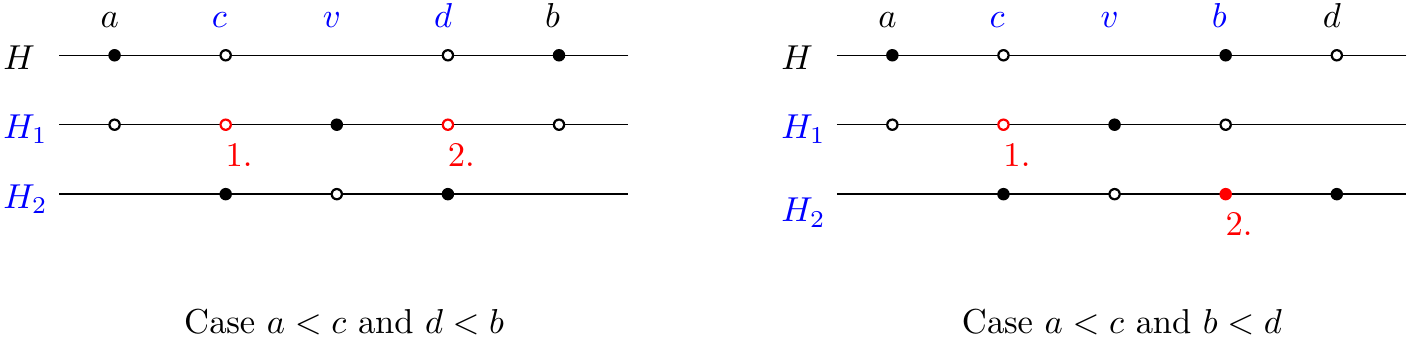}
		\caption{Lemma \ref{lem:extension}, Middle + Middle case.}
		\label{fig:extension2}
	\end{center}
\end{figure} 

\begin{figure}
	\begin{center}
		\includegraphics{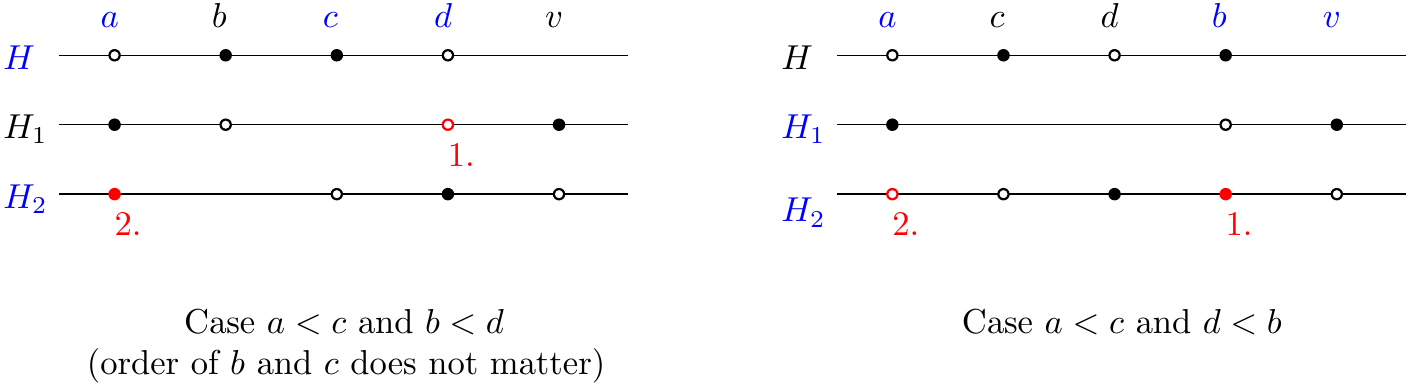}
		\caption{Lemma \ref{lem:extension}, Side + Same Side case.}
		\label{fig:extension3}
	\end{center}
\end{figure} 

\begin{figure}
	\begin{center}
		\includegraphics{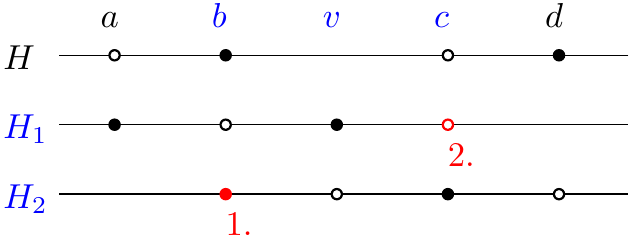}
		\caption{Lemma \ref{lem:extension}, Side + Opposite Side case.}
		\label{fig:extension4}
	\end{center}
\end{figure} 

We note that this extension is usually not unique (e.g., when $\HH'$ and $\HH'^+$ are both empty and $\HH$ contains only one hyperedge).

Applying Lemma \ref{lem:extension} to the dual of the underlying ABA-free hypergraph $\FF$, which is also ABA-free, we get (note that we can get one lemma from the other by reversing the roles of vertices and hyperedges):

\begin{lem}[Extension with a hyperedge]\label{lem:extensiondual}
	Given a pseudohalfplane hypergraph $\HH$ on vertex set $S$ and a subset $S'\subseteq S$ of the vertices. Let $\HH'$ be the subhypergraph induced by $S'$. Suppose that we can add a hyperedge $H'$ to $\HH'$ so that $\HH'\cup\{H'\}$ is a pseudohalfplane hypergraph as well. Then we can extend $H'$ to some $H$ on $S$ (i.e., $H\cap S'=H'$) so that $\HH\cup\{H\}$ is a pseudohalfplane hypergraph as well.
\end{lem}

Using these extension lemmas we can already prove the following discrete version of Levi's enlargement lemma: 

\begin{lem}[Discrete Levi's enlargement lemma]\label{lem:discretelevi}
	Given a pseudohalfplane hypergraph $\HH$ and a pair of its vertices, $p$ and $q$. Let $\HH'$ be the hypergraph we get by adding a new vertex $p'$ next to (i.e., immediately after or before) $p$ and a new vertex $q'$ next to $q$ in the vertex-order. We add $p'$ (resp. $q'$) to a hyperedge if and only if it contains $p$ (resp. $q$). Then there exists a subset $X$ of the vertices such that $X\cap \{p,p',q,q'\}=\{p,q\}$ and $\HH'\cup\{X\}$ is a pseudohalfplane hypergraph.
\end{lem}

\begin{proof}
	The theorem follows easily from Lemma \ref{lem:extensiondual}. Let $\HH^-$ be the induced subhypergraph that $\HH'$ induces on $\{p,p',q,q'\}$. Observe that every hyperedge either contains none of $p$ and $p'$ or both of them. The same holds for $q$ and $q'$. Thus if we add $\{p,q\}$ to $\HH^-$ to get the hypergraph $\HH^-_X$ then $\HH^-_X$ is trivially an ABA-free hypergraph and so it is also a pseudohalfplane hypergraph. Now we can apply Lemma \ref{lem:extensiondual} on $\HH'$ to get the subset of the vertices $X$ such that $X\cap \{p,p',q,q'\}=\{p,q\}$ and $\HH'\cup\{X\}$ is a pseudohalfplane hypergraph, as required.
\end{proof}

\begin{remark}
	While Lemma \ref{lem:extensiondual} provided an efficient way to construct an $X$ to prove Lemma \ref{lem:discretelevi}, in case $\HH$ is maximal there is also a more intuitive way to define $X$. Namely, if $\HH$ is maximal then every vertex different from $p$ and $q$ is strictly above or below $pq$ (for the definition of strictly above and below see Definition \ref{def:orientation}.). Let $X$ be the set of vertices that are strictly above $p$ and $q$ in $\HH$, plus $p$ and $q$. The proof that this choice of $X$ is as required is a case analysis similar to the one in the proof of Lemma \ref{lem:extension} and is left to the interested reader. This also gives a second proof for Lemma \ref{lem:discretelevi}, as by greedily adding hyperedges one can make $\cH$ maximal and then the $X$ which is as required for this maximal hypergraph must be also as required for $\cH$.
\end{remark}

The exact argument why this can be regarded as a discrete version of Levi's well-known lemma can be found in Section \ref{sec:geom}. Roughly speaking, because in our setting there is no direct notion of a vertex being on the boundary of a hyperedge, as a workaround we can do the following. We can duplicate a vertex and then if a hyperedge contains exactly one of the copies of the vertex then it behaves as one whose boundary passes through the original vertex. In the above theorem we do this for both $p$ and $q$ and so we get a hyperedge that behaves as one whose boundary passes through both $p$ and $q$.

\subsection{Properties of the extremal vertices}

In the remainder of this section we are always given a pseudohalfplane hypergraph $\cH\subseteq \FF\cup\bar{\FF}$ on vertex set $S$ whose set of extremal vertices is denoted by $E(\HH)$. We recall several statements from \cite{kbdiscretehelly}. First, the following observation provides an equivalent definition for the extremal vertices:

\begin{claim}\label{claim:singleton} \cite{kbdiscretehelly}
	The topvertices of $\cH$ are exactly those vertices $v$ for which if we add the singleton hyperedge $\{v\}$ as a topset to $\FF\cup\bar{\FF}$, i.e. we add it to $\FF$, we still get a pseudohalfplane hypergraph. The bottomvertices are exactly those vertices $v$ for which if we add the singleton hyperedge $\{v\}$ as a bottomset to $\FF\cup\bar{\FF}$, i.e. we add it to $\bar{\FF}$ (that is, we add $S\setminus \{v\}$ to $\FF$), we still get a pseudohalfplane hypergraph. 
	
	In other words, the extremal vertices are exactly those vertices that can be separated from the rest of the vertices by a (possibly additional) hyperedge.\footnote{Note that in the geometric setting of halfplanes this wording would give the extreme vertices instead of the vertices that lie on the boundary of the convex hull.}
\end{claim} 

%
%
%

%
\begin{obs} \cite{kbdiscretehelly}
	The leftmost and rightmost vertices of $\cH$ are both topvertices and bottomvertices and so they are always extremal vertices. 
\end{obs}
\begin{claim} \cite{kbdiscretehelly}
	Every hyperedge of $\cH$ intersects the set of extremal vertices.
\end{claim}
\begin{claim} \label{convexsizeisminthree}\cite{kbdiscretehelly}
	If the vertex set $S$ has size $n\ge 3$, then the extremal vertex set contains at least $3$ vertices.
\end{claim}
%

%
\begin{obs}[Topvertices in a topset]\label{obs:topintop}  \cite{kbdiscretehelly}
	If $X$ is a topset and $x,y\in X$, then $X$ contains all topvertices that are between $x$ and $y$. The same holds with bottomvertices if $X$ is a bottomset.
\end{obs}
\begin{obs}[Bottomvertex in a topset]\label{obs:topindown}  \cite{kbdiscretehelly}
	If $X$ is a topset and $x\in X$ is a bottomvertex, then $X$ contains all vertices that are bigger or all vertices that are smaller than $x$. The same holds if $X$ is a bottomset and $x\in X$ is a topvertex.
\end{obs}
Let $T=(t_1=v_1,t_2,\dots, t_k=v_n)$ and $B=(b_1=v_1,b_2,\dots, b_l=v_n)$ be the sets of top and bottom vertices ordered according to the ordering on $S$. Call $T$ to be the upper hull and $B$ the lower hull. Note that a vertex may appear in both sets. Let us give the following circular order on $E(\HH)$, the set of extremal vertices: $(v_1,t_2,\dots, t_{k-1},v_n, b_{l-1},\dots, b_2)$\footnote{This circular order corresponds to the clockwise order of points on the convex hull in the geometric case defined by halfplanes. Also, $T$ corresponds to the upper hull and $B$ to the lower hull in the geometric case. Note that a vertex may appear twice in this circular order of $E(\HH)$.}.

\begin{lem}\label{lem:hullinterval} \cite{kbdiscretehelly}
	Every hyperedge of $\HH$ intersects the extremal vertex set in an interval of the circular order defined on the extremal vertex set.
\end{lem}

\begin{lem}\label{lem:consecutives} \cite{kbdiscretehelly}
	If a topset (resp. bottomset) $H\in \cH$ contains two bottomvertices (resp. topvertices) $p<q$ that are consecutive in the circular order of the extremal vertices, then $H$ contains every vertex $r$ with $p<r<q$.
\end{lem}

\begin{claim}\label{claim:everyconsecutive} \cite{kbdiscretehelly}
	If a topset (resp. bottomset) $H\in \cH$ contains every bottomvertex (resp. topvertex) then it contains every vertex.
	
	If the hyperedge $H\in \cH$ contains every extremal vertex then it contains every vertex.
\end{claim}

\begin{cor}\label{cor:convofextremal}
	Given a pseudohalfplane hypergraph $\HH$ on vertex set $S$, $Conv(E(\HH))=S$.
\end{cor}

While for simplicity we defined only the extremal vertices of the whole vertex set $S$ of the hypergraph $\HH$ (and denoted it by $E(\HH)$), we could naturally extend this and define the extremal vertices of a subset $S'$ of the vertices to be $E(\HH[S'])$. Corollary \ref{cor:convofextremal} immediately implies the following formal strengthening which is the Krein-Milman theorem for pseudoconvex sets:

\begin{cor}\label{cor:convofextremalsub}
	Given a pseudohalfplane hypergraph $\HH$ on vertex set $S$ and a subset $S'\subseteq S$ of the vertices, $Conv(E(\HH[S']))=Conv(S')$ in $\HH$, i.e., for any subset $S'$ its convex hull coincides with the convex hull of its extremal vertices.
\end{cor}

Indeed, Corollary \ref{cor:convofextremal} implies $Conv(E(\HH[S']))\supseteq S'$ while $Conv(E(\HH[S']))\subseteq Conv(S')$ follows from $E(\HH[S'])\subseteq S'$ as $Conv$ is monotone for containment by definition.

We can define an above/below relation among triples of vertices in the following way:

\begin{defi}\label{def:orientation}
	Given three vertices $a<b<c$ in a pseudohalfplane hypergraph $\HH$, we say that 
	$b$ is \emph{above} $ac$, $c$ is \emph{below} $ab$ and $a$ is below $bc$ (resp. $b$ is below $ac$, $c$ is above $ab$ and $a$ is above $bc$) if on the induced subhypergraph defined by these three vertices $b$ is a topvertex (resp. bottomvertex). A vertex is \emph{strictly above} (resp.\ \emph{strictly below}) another two if it is above them but not below (resp.\ below them but not above).
\end{defi}

We note that by Claim \ref{convexsizeisminthree} at least one of the cases in the above definition holds.

\begin{obs}
	Given three vertices $a<b<c$ in a pseudohalfplane hypergraph $\HH$, $b$ is both above and below $ac$ if and only if $b\in Conv(\{a,c\})$.
\end{obs}

Note that if the hypergraph is defined by all halfplanes on a planar set of points then a vertex is below a pair of vertices if and only if it is below or on the line that goes through these two vertices. A vertex is strictly below a pair of vertices if and only if it is below the line that goes through these two vertices.

Next we show two additional interesting properties of the extremal vertices. Notice that these statements could have been phrased as statements about ABA-free hypergraphs as they concern only $\FF$, the family of topsets, which is ABA-free, in particular the respective claims where we replace topvertices with bottomvertices also hold.

\begin{lem}\label{lem:deletetop}
	Given a vertex $v$ that is not a topvertex. Then deleting $v$ does not change the set of topvertices, i.e., $T(\cH[S\setminus\{v\}])=T(\cH)$.
\end{lem}

\begin{proof}
	Assume on the contrary that by deleting $v$ a vertex $w$ becomes a topvertex, wlog. assume that $v<w$. As $w$ was not a topvertex in $\HH$, there exists a hyperedge $H\in \FF$ that skips $w$ in $\HH$ but does not skip $w$ after deleting $v$. This implies that $v$ must be the only vertex in $H$ before $w$ and that $H$ avoids $w$. 
	Now using that $v$ is not a topvertex in $\cH$, it is skipped by a hyperedge $H'\in \FF$. Thus $H'$ contains vertices $v',v''$ and avoids $v$ such that $v'<v<v''$. If $H'$ contains $w$ then $H$ and $H'$ form an ABA-occurrence on the vertices $v',v,w$, a contradiction (using that $H$ contains no vertex before $v$, thus $v'\notin H$). If $H'$ contains a vertex between $v$ and $w$ then similarly we get an ABA-occurrence. Thus $w<v''$ and so $H'$ skips $w$ on $S\setminus\{v\}$, thus $w$ is not a topvertex after deleting $v$, a contradiction.
\end{proof}

\begin{lem}\label{lem:conseqtop}
	Given a vertex $v$ that is not a topvertex. Let $t_i<v<t_{i+1}$ be the topvertices right before and right after $v$. Then there exists a topset $H\in \FF$ such that $v\notin H$ while $t_i,t_{i+1}\in H$.
\end{lem}

\begin{proof}
	We prove the statement by induction on the number of vertices between $t_i$ and $t_{i+1}$. Assume on the contrary that $v$ is not a topvertex yet the required $H$ does not exist. As $v$ is not a topvertex, there is a hyperedge $H'\in \FF$ that skips $v$. If it contains a vertex before $t_i$ then by Observation \ref{obs:topintop} it also contains $t_i$. Similarly, if it contains a vertex after $t_{i+1}$ then it also contains $t_{i+1}$. It cannot contain only vertices between $t_{i}$ and $t_{i+1}$ as then using that every topset contains a topvertex, we would get a topvertex between $t_{i}$ and $t_{i+1}$, a contradiction. Thus $H'$ must contain exactly one of these two vertices, wlog. contains $t_{i}$, avoids $t_{i+1}$ and contains another vertex $w$ s.t. $v<w<t_{i+1}$.
	
	Now we temporarily delete $v$, by Lemma \ref{lem:deletetop} $w$ is still not a topvertex. As there is one less vertex between $t_i$ and $t_{i+1}$, we can apply induction to find a hyperedge $H''\in \FF$ such that $H''\setminus \{v\}$ avoids $w$ yet contains $t_i$ and $t_{i+1}$. If $v\notin H''$ then $H=H''$ is as required. Otherwise, if $v\in H''$, then $H'$ and $H''$ form an ABA-occurrence on $v,w,t_{i+1}$, a contradiction.	
\end{proof}

\section{Proofs}

\subsection{Proofs of Helly theorems about pseudoconvex sets}\label{sec:proofconvex}

\begin{proof}[Proof of Theorem \ref{thm:weakhellypsconvex}]
	Given a pseudohalfplane hypergraph $\HH$ and a subfamily $\CC$ of its convex sets such that every triple of convex sets from $\CC$ has a common vertex, we claim that we can add a vertex contained in every convex set of $\CC$ (so that the new hypergraph is still a pseudohalfplane hypergraph).
	
	For each $C\in \CC$ let $\HH_C\subseteq \HH$ be the family of pseudohalfplanes such that $\cap \{H: H\in \HH_C\}=C$. Let $\HH_\CC=\cup \{\HH_C:C\in \CC\}$. As $\HH_\CC\subseteq \HH$, it is a pseudohalfplane hypergraph. Lemma \ref{lem:weakhellypshp} implies that we can add a new vertex $v$ to $\HH_\CC$ and extend its hyperedges appropriately so that $v$ is in every hyperedge of $\HH_\CC$ and the new hypergraph $\HH_\CC^+$ is still a pseudohalfplane hypergraph. Finally, we can apply Lemma \ref{lem:extension} to conclude that one can extend also $\HH$ to $v$ to get $\HH^+$ so that $\HH_\CC^+\subseteq \HH^+$. Here $v$ is in every hyperedge of $\HH_\CC^+$ and thus in every convex set of $\CC^+$, as required (where each set of $\CC^+$ is an original convex set from $\CC$ plus the vertex $v$, which is still convex as it is the intersection of the hyperedges of $\HH_\CC^+$).
\end{proof}

\begin{proof}[Proof of Theorem \ref{thm:weakhellypshsconvex}]
	Given a pseudohemisphere hypergraph $\HH$ and a subfamily $\CC$ of its convex sets such that every $4$-tuple of convex sets from $\CC$ has a common vertex, we claim that we can add a vertex contained in every convex set of $\CC$ (so that the new hypergraph is still a pseudohemisphere hypergraph).
	
	We only sketch the proof, the details are left to the interested reader. Using Lemma \ref{lem:extension} one can easily prove a similar extension lemma for pseudohemisphere hypergraphs. Similar to the proof of Theorem \ref{thm:weakhellypsconvex} applying Lemma \ref{lem:weakhellypshs} and then this extension lemma we get the required statement.
\end{proof}

The reason why we need $4$-tuples in Theorem \ref{thm:weakhellypshsconvex} is that
in Lemma \ref{lem:weakhellypshs} for pseudohemisphere hypergraphs we need that every $4$-tuple intersects (and this is optimal according to Claim \ref{claim:no2-1}), whereas in Lemma \ref{lem:weakhellypshp} for pseudohalfplane hypergraphs we just needed that every $3$-tuple intersects.

\begin{proof}[Proof of Theorem \ref{thm:dualweakpshs}]
Applying Theorem \ref{thm:weakhellypshsconvex} on the dual hypergraph, which is also a pseudohemisphere hypergraph (as shown in \cite{abafree}), where $\CC$ is set to be the family of hyperedges that are duals of the points in $S'$, implies the required statement.
\end{proof}

As we have already mentioned earlier, in \cite{kbdiscretehelly} considering (primal) discrete Helly theorems for pseudohalfplane hypergraphs, it was proven that we can even guarantee the existence of a bounded number of vertices of $S$ that hit all hyperedges, that is, we do not need new vertices to hit all hyperedges. E.g., Theorem \ref{thm:primalpshp32} states that Theorem \ref{lem:weakhellypshp} can be modified such that we require instead of one new vertex (not necessarily from $S$) that hit all hyperedges two vertices from $S$ that together hit all hyperedges. However, a simple construction about halfplanes of \cite{jjr} shows that for pseudoconvex sets one cannot expect such a strong discrete Helly theorem where the hitting vertices are from $S$. 

Also, in \cite{kbdiscretehelly} dual strong discrete Helly theorems were proved for pseudohalfplanes. E.g., if in a family $\HH$ of pseudohalfplanes on vertex set $S$ every subset of $3$ vertices in $S$ belongs to some hyperedge $H\in \cH$ then there exists two hyperedges in $\cH$ whose union covers $S$. However, as in the primal case, for pseudoconvex sets such a dual strong discrete Helly theorem cannot hold:

\begin{claim}\label{claim:nodualstrong}
	For every $c,m$ there exists a pseudohalfplane hypergraph $\HH$ on vertex set $S$ and a subfamily $\CC$ of its convex sets such that every subset of $c$ vertices belongs to some convex set of $\CC$ but there do not exist $m$ convex sets in $\CC$ whose union covers $S$.
\end{claim}

\begin{proof}
	Just put $cm+1$ vertices in convex position and let $\HH$ be the pseudohalfplane hypergraph defined by every halfplane. Then every subset of the vertices is a convex set. Let $\CC$ be the family of all the sets of size $c$. This family has the properties required by the claim.
\end{proof}

\subsection{Proofs of generalizations of further classical results}\label{sec:proofclassical}

\begin{proof}[Proof of Claim \ref{claim:strongconvex}]
	Suppose first that $v$ is not strongly inside the convex hull of $S'$ in $\HH$, i.e., there exists an extension $\HH'$ of $\HH$ in which $v\notin Conv(S')$. This means that there is a hyperedge $H$ in $\HH'$ that contains $S'$ but does not contain $v$. One can then add the complement $\bar H$ to $\HH'$ as well. Restricted to $S'\cup\{v\}$ we have $\bar H\cap (S'\cup\{v\})=\{v\}$ and thus by Claim \ref{claim:singleton} in $\HH$ we have that $v$ is extremal in $S'\cup\{v\}$.
	
	Suppose now that $v$ is extremal in $S'\cup\{v\}$. Again by Claim \ref{claim:singleton} when restricted to $S'\cup\{v\}$, $\HH$ can be extended with the hyperedge $\{v\}$. Using Lemma \ref{lem:extensiondual} we can extend this hyperedge such that we get a new hyperedge $H$ such that $\HH$ with $H$ is a pseudohalfplane hypergraph on $S$ and $H\cap (S'\cup\{v\})=\{v\}$. We can add $\bar H$ to the hypergraph as well to get the hypergraph $\HH'$ for which $\bar H\cap (S'\cup\{v\})=S'$ and then by the definition of the convex hull we get that in $\HH'$ we have $v\notin Conv(S')$, i.e., $v$ is not strongly inside the convex hull of $S'$ in $\HH$.
\end{proof}

Next we proceed by proving the two versions of Carath\'eodory's Theorem for pseudoconvex sets.

\begin{proof}[Proof of Claim \ref{claim:pseudosteinitz}]
We are given a pseudohalfplane hypergraph $\HH$ with $\HH\subseteq \FF\cup \bar{\FF}$, wlog. $\HH=\FF\cup \bar{\FF}$, on vertex set $S\cup\{v\}$ where $\FF$ is an ABA-free hypergraph. Let $\HH'=\HH[S'\cup\{v\}]$. We assume that $v$ is not an extremal vertex of $S'\cup\{v\}$. Let $T=(t_1=v_1,t_2,\dots, t_k=v_n)$ and $B=(b_1=v_1,b_2,\dots, b_l=v_n)$ be the sets of top and bottom vertices of $\HH'$ ordered according to the ordering on $S\cup\{v\}$. Now there exist $i$ and $j$ such that $t_i<v<t_{i+1}$ and $b_j<v<b_{j+1}$. We claim that for $S''=\{t_i,t_{i+1},b_j,b_{j+1}\}$ we have $v\in Conv(S'')$ in $\HH'$ and thus also in $\HH$. To see this, it is enough to prove that for an arbitrary hyperedge $H$ of $\HH'$ containing $S''$, $H$ also must contain $v$. If $H$ is a topset in $\HH'$ then by Lemma \ref{lem:consecutives} $b_j,b_{j+1}\in H$ implies $v\in H$ and if $H$ is a bottomset in $\HH'$ then by Lemma \ref{lem:consecutives} $t_i,t_{i+1}\in H$ implies $v\in H$. 

We are left to prove that $v$ is strongly inside the convex hull of $S''$.
As $v$ is strongly inside the convex hull of $S$, by Claim \ref{claim:strongconvex} it is not a topvertex of $\HH$ and thus using Lemma \ref{lem:conseqtop} we get that there is a topset that contains $t_i$ and $t_{i+1}$ but avoids $v$. This topset shows that $v$ is not a topvertex in $\HH[S'']$. The same way get that $v$ is not a bottomvertex in $\HH[S'']$ and thus by Claim \ref{claim:strongconvex} we get that $v$ is strongly inside the convex hull of $S''$, as required.
\end{proof}

\begin{proof}[Proof of Theorem \ref{thm:pseudocaratheodory}]
We are given a pseudohalfplane hypergraph $\HH$, wlog. $\HH=\FF\cup \bar{\FF}$ on vertex set $S\cup\{v\}$ where $\FF$ is an ABA-free hypergraph.  Let $\HH'=\HH[S'\cup\{v\}]$. Let $T=(t_1=v_1,t_2,\dots, t_k=v_n)$ and $B=(b_1=v_1,b_2,\dots, b_l=v_n)$ be the sets of top and bottom vertices of $\HH'$ ordered according to the ordering on $S\cup\{v\}$ and $C=(c_1=t_1,c_2=t_2,\dots, c_k=t_k=b_l,c_{k+1}=b_{l-1},\dots, c_n=b_2)$ the circular order of the extremal vertices. We assume that $v$ is not an extremal vertex of $S'\cup\{v\}$ and we want to find at most three vertices of $S'$ such that their convex hull in $\HH'$ (and thus also in $\HH$) already contains $v$. 

If $v$ is contained in the convex hull (in $\HH'$) of two other vertices of $S'$ then we are done. Suppose from now on that this is not the case. Thus, if $v$ is above (resp. below) a pair of vertices in $\HH'$ then it cannot be below (resp. above) them and so it is strictly above (resp. below) them.

We aim to find two extremal vertices of $\HH'$, $c_i$ and $c_{i+1}$, consecutive in the circular order of $E(\HH')$ such that $v\in Conv(v_1,c_i,c_{i+1})$ in $\HH'$. Note that in the geometric setting of halfplanes this corresponds to the point being in a triangle of the triangulation of the convex hull in which every triangle is incident to the leftmost point.


We need the following observation: if $p$ is an extremal vertex and $p<v$ then if $p$ is a topvertex (resp. bottomvertex) then $v$ is below (resp. above) $v_1p$.



We claim that $v$ must be (strictly) below $v_1t_2$ and similarly $v$ must be (strictly) above $v_1b_2$. It is enough to prove the first, the other case can be done the same way. If $v$ is to the right of $t_2$ then $v$ is below $v_1t_2$ by the previous observation. Otherwise $v$ is between $v_1$ and $t_2$. By Lemma \ref{lem:consecutives} there cannot exist a bottomset such that it contains $v_1,t_2$ but does not contain $v$, thus $v$ is again below $v_1t_2$.

As $v$ is (strictly) below $v_1c_2$ but (strictly) above $v_1c_n$, we can take in the circular order a $c_i$ ($i\ne n$) such that $v$ is strictly below $v_1c_i$  and $v$ is strictly above $v_1c_{i+1}$. Note that as the rightmost vertex is both a topvertex and a bottomvertex, $c_i$ and $c_{i+1}$ are either both topvertices or both bottomvertices, wlog. they are topvertices and so $c_i<c_{i+1}$. We claim that $v\in Conv(v_1,c_i,c_{i+1})$. To prove this, take an arbitrary hyperedge $H$ which contains $v_1,c_i,c_{i+1}$, we need to prove that it contains $v$ as well. First, $v<c_{i+1}$ as otherwise $v$ would be below $vc_{i+1}$, a contradiction. Second, if $H$ avoids $v$ and $v<c_i$ then if $H$ is a topset then $v$ is below $v_1c_{i+1}$ and if $H$ is a bottomset then $v$ is above $vc_i$, both are contradictions. Third, if $c_i<v<c_{i+1}$ then if $H$ avoids $v$ then if $H$ is a topset then $v$ is below $v_1c_{i+1}$, a contradiction and if $H$ is a bottomset then using Lemma \ref{lem:consecutives} $H$ containing $c_i$ and $c_{i+1}$ must also contain $v$, a contradiction.
\end{proof}

Notice that if $v$ is not extremal in $S'\cup\{v\}$ then by Claim \ref{claim:strongconvex} $v$ is strongly inside the convex hull of $S'$ and thus $v\in Conv(S')$ in $\HH$. 
We claim that if the pseudohalfplane hypergraph $\HH$ is maximal then the reverse is also true. It is enough to prove that if $v$ is extremal in $S'\cup\{v\}$, then $v\notin Conv(S')$. In this case, by Claim \ref{claim:singleton} we can add the hyperedge $S'$ to $\HH[S'\cup\{v\}]$. Using Lemma \ref{lem:extensiondual} we can also add a hyperedge to $\HH$ that coincides with this hyperedge on $S'\cup\{v\}$. As $\HH$ is maximal, this hyperedge must already be in $\HH$ which shows that $v\notin Conv(S')$.

\begin{lem}\label{lem:4allextremal}
	If in a pseudohalfplane hypergraph $\HH$ on vertex set $S'=\{a,b,c,d\}$ with $a<b<c$ and $a<d<c$ we have that $b$ is above $ac$ and $d$ is below $ac$ then  all vertices of $S'$ are extremal, $b$ is a topvertex and $d$ is a bottomvertex.
\end{lem}

\begin{proof}
	Wlog. $b<d$. In $S'$ besides $a$ and $c$ there must be at least one more extremal vertex, wlog. $d$. As $d$ is below $ac$, $d$ must be a bottomvertex in $S'$. We claim that $b$ must be a topvertex in $S'$. Indeed, otherwise there would be a topset $F$ in $\HH$ that skips $b$ on $S'$, that is, avoids $b$ but contains $a$ and at least one of $c$ and $d$. If $F$ would contain $c$ then $b$ would not be above $ac$, a contradiction. If $F$ would contain $d$ then $F$ must also contain $c$ by Observation \ref{obs:topindown} (using that it does not contain $b$) and then it is again a contradiction as before.
\end{proof}

\begin{lem}\label{lem:4vertices}
	If in a pseudohalfplane hypergraph $\HH$ on vertex set $S$ there is a subset $S'=\{a,b,c,d\}$ such that in the subhypergraph induced by $S'$ all vertices of $S'$ are extremal and in the circular order they appear in order $a,b,c,d$, then we can extend $\HH$ with a new vertex $v$ such that $v\in Conv(\{a,c\}\cap Conv(\{b,d\})$ in this extended hypergraph.
\end{lem}

\begin{proof}
	Let $\HH_{a,c}\subseteq \HH$ denote the family of hyperedges that contain both of $a,c$ and $\HH_{b,d}\subseteq \HH$ denote the family of hyperedges that contain both of $b,d$. Notice that by Lemma \ref{lem:hullinterval} every hyperedge in $\HH_{a,c}$ must contain also either $b$ or $d$ (or both) and similarly every hyperedge in $\HH_{b,d}$ must contain also either $a$ or $c$ (or both). Thus every hyperedge in $\HH'=\HH_{a,c}\cup\HH_{b,d}$ contains at least $3$ of the $4$ vertices and so every triple of hyperedges from $\HH'$ must have a common vertex. We can apply Theorem \ref{thm:weakhellypsconvex} and add a vertex $v$ to $\HH$ to get the pseudohalfplane hypergraph $\HH^+$ such that $v$ is contained in all hyperedges of $\HH'^+$, the subhypergraph containing the extensions of the hyperedges of $\HH'$. Finally, notice that if $v$ is in every hyperedge that contains both of $a,c$ or both of $b,d$ then $v$ is in the convex hull of $\{a,c\}$ and of $\{b,d\}$ in this extended hypergraph $\HH^+$, finishing the proof.
\end{proof}

\begin{proof}[Proof of Theorem \ref{thm:multiequi}]
	First, $(1)\rightarrow (2)$ follows from the fact that $Conv(A\cap D)\subseteq Conv(A)$ and $Conv(B\cap D)\subseteq Conv(B)$.
	
	Second, to see $(3)\rightarrow (1)$ let the set guaranteed by $(3)$ be $H$. We can extend $\HH$ with $H$ and $\bar H$ to get the pseudohalfplane hypergraph $\HH'$ in which $H$ contains $A$ but avoids $B$ and $\bar H$ contains $B$ but avoids $A$. As in any extension by a vertex $v$ by definition of an extension $v$ is in exactly one of $H'$ and $\bar H'$ (the hyperedges corresponding to $H$ and $H'$ in the extended hypergraph), we get that $v\notin Conv(A)$ or $v\notin Conv(B)$, thus $(1)$ follows.
	
	We are left to prove $(2)\rightarrow (3)$, which will take some more effort.

	We assume that there exists an extension $\HH'$ of $\HH$ with additional hyperedges such that $\HH'$ cannot be extended with a new vertex $v$ such that for some subset $D\subseteq S$ with $|D|\le 4$ we have $v\in Conv(A\cap D)\cap Conv(B\cap D)$ in $\HH'$.

	We can assume that $\HH'$ contains $\emptyset$ as a hyperedge and so $Conv(\emptyset)=\emptyset$ in $\HH'$.

	
	Notice that for every $b<a<b'$, $a\in A$ and $b,b'\in B$, if $a$ is both above and below $bb'$ in $\HH'$ then $a$ is in the convex hull of $\{b,b'\}\subset B$, a contradiction by setting $D=\{b,a,b'\}$ and $v=a$.
	Thus we can assume that $a$ is strictly above or below $bb'$ in $\HH'$. Similarly, for every $a<b<a'$, $b$ is strictly above or below $aa'$ in $\HH'$.	

	We claim that for every vertex $a\in A$ and every pair of vertices $b,b'\in B$ such that $b<a<b'$ we either have that $a$ is strictly above $bb'$ for all such triples or strictly below $bb'$ for all such triples. To see this, assume on the contrary that there is a triple $b<a<b'$ where $a$ is strictly above $bb'$ and another triple $d<c<d'$ where $c$ is strictly below $dd'$ ($a,c\in A$ and $b,b',d, d'\in B$). Now let $l=\min(b,b',d,d')$ and $r=\max (b,b',d,d')$. We claim that $a$ is strictly above $lr$. Assume on the contrary, $a$ is strictly below $lr$. Then $a$ cannot be extremal in $S'=\{a,b,b',d,d'\}$ (being strictly above $bb'$ and below $lr$), and so by Theorem \ref{thm:pseudocaratheodory} $a$ is in the convex hull of at most three vertices from $\{b,b',d,d'\}\subset B$, a contradiction by setting $D$ to be the union of these at most three vertices and $\{a\}$ and setting $v=a$. Similarly, $c$ is strictly below $lr$. Then using Lemma \ref{lem:4allextremal} and Lemma \ref{lem:4vertices} we get that one can add a new vertex $v$ in $Conv(\{a,c\})\cap Conv(\{l,r\})$, which is a contradiction by setting $D=\{a,c,l,r\}$ (note that $l,r\in B$ and $a,c\in A$).
		
	Similarly, for every vertex $b\in B$ and every pair of vertices $a,a'\in A$ such that $a<b<a'$ we either have that $b$ is strictly above $aa'$ for all such triples or strictly below $aa'$ for all such triples.
	
	If one type of triples does not exist (triple $b<a<b'$ or $a<b<a'$ with $a,a'\in A$ and $b,b'\in B$) then wlog. we assume that there is no triple of type $a<b<a'$.
	
	Wlog. for every $b<a<b'$ with $a\in A$ and $b,b'\in B$, we have that $a$ is strictly above $bb'$. We claim that then for every $a<b<a'$ with $a,a'\in A$ and $b\in B$, we have that $b$ is strictly below $aa'$. If there is no such triple then we are trivially done. Otherwise, there exists a set of four vertices $S'=\{a,b,a',b'\}$ with $a<b<a'<b'$ alternately in $A$ and $B$, wlog. $a,a'\in A$ and $b,b'\in B$. It is enough to show that $b$ is strictly below $aa'$ (as then the same holds for every other such triple). Assume on the contrary that $b$ is strictly above $aa'$ and recall that $a'$ is strictly above $bb'$. In $S'$ besides $a$ and $b'$ there must be at least one more extremal vertex, wlog. $a'$, as it is above $bb'$, $a'$ must be a topvertex in $S'$. As $b$ is strictly above $aa'$, it cannot be a bottomvertex in $S'$. We claim that it must be a topvertex in $S'$. Indeed, otherwise there is a topset $F$ in $\HH$ that skips $b$ on $S'$, that is, $F$ avoids $b$ but contains $a$ and at least one of $a'$ and $b'$. If $F$ would contain $a'$ then $b$ would not be strictly above $aa'$, a contradiction. If $F$ would contain $b'$ then as $a'$ is a topvertex of $S'$, it must also contain $a'$ by Observation \ref{obs:topintop} and then it is again a contradiction as before. Thus, $b$ is a bottomvertex and then by Lemma \ref{lem:4vertices} we can add a vertex $v$ in $Conv(\{a,a'\})\cap Conv(\{b,b'\})$, a contradiction.
		
	Now let $\HH^-=\HH'[A\cup B]$. Add the hyperedge $F=A$ to $\HH^-$ as a topset, we claim that it remains to be a pseudohalfplane hypergraph (notice that $F$ separates $A$ and $B$). Assume on the contrary that there is a topset $F'\in \HH'$ such that with $F$ they have an ABA-occurrence on $A\cup B$. This is possible in two ways. First, if we have a subset $S'=\{a,b,a'\}$, $a<b<a'$ such that $a,a'\in A$ and $b\in B$ and so $F\cap S'=\{a,a'\}$ while $F'\cap S'=\{b\}$, then $b$ would be a topvertex in $S'$, that is, $b$ would be above $aa'$, a contradiction. In the second case we have a subset $S'=\{b,a,b'\}$, $b<a<b'$ such that $a\in A$ and $b,b'\in B$ and so $F\cap S'=\{a\}$ while $F'\cap S'=\{b,b'\}$, then $a$ would be a bottomvertex in $S'$, that is, $a$ would be below $bb'$, a contradiction. 
	
	Finally, using Lemma \ref{lem:extensiondual} we can extend $\HH'$ with $F$ to be a new hyperedge, which separates $A$ and $B$. As $\HH\subseteq \HH'$, we can also extend $\HH$ with $F$.
\end{proof}

We note that in the proof of Theorem \ref{thm:multiequi} we used Theorem \ref{thm:pseudocaratheodory} once for a subset of size $5$ so even though Theorem \ref{thm:multiequi} implies Theorem \ref{thm:pseudocaratheodory}, its proof did not become redundant.

\begin{proof}[Proof of Theorem \ref{thm:kirch}]
	Given a pseudohalfplane hypergraph $\HH$ on vertex set $S$ and subsets $A \cap B= \emptyset$ such that for every subset $D\subseteq S$ with $|D|\le 4$ there exists a hyperedge of $\HH$ separating $A\cap D$ and $B\cap D$. Let $\HH'$ be the extension of $\HH$ such that for each such separating hyperedge we also add its complement to $\HH'$ if it is not already in $\HH$. Then by definition of the convex hull, in $\HH'$ for every such $D$ we have $Conv(A\cap D)\cap Conv (B\cap D)=\emptyset$. We can thus apply Theorem \ref{thm:multiequi} $(2)\rightarrow (3)$ to conclude that there exists an extension $\HH''$ of $\HH$ with one new hyperedge $H$ such that $H$ separates $A$ and $B$, as required.	
\end{proof}

\begin{proof}[Proof of Theorem \ref{thm:radon}]
	Take a subset $S'\subseteq S$ with $|S'|=4$ of the vertex set $S$ of the pseudohalfplane hypergraph $\HH$. If there is one vertex of $S'$ that is not an extremal vertex of $\HH[S']$ then by Corollary \ref{cor:convofextremalsub} the non-extremal vertex is contained in the convex hull of the other $3$ vertices and we are done. 
	Thus we can suppose that all $4$ vertices are extremal in $\HH[S']$, suppose their circular order is $a,b,c,d$. By Lemma \ref{lem:4vertices} we can add a new vertex $v$ to $\HH$ such that $v$ is in the convex hull of $\{a,c\}$ and of $\{b,d\}$ in this extended hypergraph, finishing the proof.
\end{proof}

\subsection{Proof of the Cup-Cap Theorem for pseudoconvex sets}\label{sec:proofcupcap}

\begin{lem}\label{lem:eszlemma}
	Given a pseudohalfplane hypergraph $\HH$ on vertex set $S$ and a vertex $v\in S$. Suppose that there exists a $k$-cup $A=\{a_1,a_2,\dots, a_k=v\}$ ending in $v$ and an $l$-cap $B=\{v=b_1,b_2,\dots, b_l\}$ starting at $v$, then either $A'=A\cup\{b_2\}$ is a $k+1$-cup or $B'=\{a_{k-1}\}\cup B$ is an $l+1$-cap.
\end{lem}

\begin{proof}	
	Suppose that the statement does not hold. Wlog. let $\HH=\FF\cup\bar{\FF}$ where $\FF$ is an ABA-free hypergraph. 
	
	As $A'$ is not a cup, some vertex $a_i$ is not a bottomvertex in $\HH[A']$ and so must be skipped by some hyperedge $\bar{F_A}$ ($F_A\in \FF$) on $\AA'$. $\bar{F_A}$ must also contain $b_2$ and some $a_j$, $j<i$. Notice that $\bar{F_A}$ cannot contain $v$ as restricted to $A$ no vertex is skipped by a bottomset (as $A$ is a cup). We also claim that $\bar{F_A}$ must contain every $b\in B, b\ne v$. Indeed, otherwise $F_A$ would contain $v$ and $b$ but avoid $b_2$ thus $b_2$ would not be a topvertex in $B$, contradicting that $B$ was a cap. Thus, $F_A\in \FF$ avoids $a_j$, contains $v$ and avoids every $b\in B, b\ne v$.
	
	Similarly, as $B'$ is not a cap, there exists a hyperedge $F_B\in \FF$ such that $F_B$ contains $a_{k-1}$ and some $b_{j'}$ (while avoids some $b_{i'}$, $i'<j'$) and must contain every $a\in A,a\ne v$.
	
	Thus, on vertices $a_j,v,b_{j'}$ $F_B$ and $F_A$ form an ABA-sequence, a contradiction.
\end{proof}

\begin{proof}[Proof of Theorem \ref{thm:pseudoesz}]
	Replacing in the original proof of Erd\H os and Szekeres the geometric argument by Lemma \ref{lem:eszlemma}, we can verbatim follow the rest of the original proof (see, e.g., \cite{tv_esz}) to conclude Theorem \ref{thm:pseudoesz}, this is left to the interested reader.
\end{proof}

We sketch a second proof for Theorem \ref{thm:pseudoesz}, suggested by A. Holmsen. First, it is easy to see that a set of vertices is a cup (resp. cap) if and only if every subset of three vertices is a cup (resp. cap). Second, we can extend the pseudohalfplane hypergraph greedily by adding new hyperedges such that at the end every triple of vertices is either a cap or a cup (but not both). One can check that two-coloring the triples according to this to get color classes $T_{cup}$ and $T_{cap}$, we get a transitive coloring, i.e., for which $(s_1, s_2, s_3), (s_2, s_3, s_4) \in T_i\Rightarrow (s_1, s_2, s_4), (s_1, s_3, s_4) \in T_i$ whenever $s_1<s_2<s_3<s_4, i\in\{cup,cap\}$.\footnote{This statement is in a sense complementary to Lemma \ref{lem:4allextremal} and can be proved similarly.} We can conclude the proof using the version of the Erd\H os-Szekeres Cup-Cap Theorem for transitive colorings (see \cite{eszmonotone,Hubard,moshkovitz,pseudoerdosszekeres}).

\subsection{Proofs of geometric variants}\label{sec:geom}
\bigskip

The main tool for proving implications of the abstract setting to the geometric setting is the following theorem from \cite{kbdiscretehelly} (we note that it is not true for loose pseudoline arrangements):

\begin{thm}\label{thm:unique}\cite{kbdiscretehelly}
	Given a simple pseudoline arrangement $\AA$, let $P$ be a set of points which has exactly one point in each face of $\AA$. Let $\HH$ be a pseudohalfplane hypergraph whose vertex set is $P$ and for each pseudoline of $\AA$ it has a hyperedge which contains the points on one side of this pseudoline. Then in every realization of $\HH$ with pseudohalfplanes the arrangement of the boundary pseudolines is equivalent to $\AA$.
\end{thm}

First we show that Lemma \ref{lem:discretelevi}, which we called a discrete Levi's enlargement lemma, indeed implies the original Levi's enlargement lemma for pairs of points that do not lie on any of the pseudolines, that is, given a pseudoline arrangement $\AA$ and points $p$ and $q$ that do not lie on any of the pseudolines, we can add a new pseudoline to the arrangement that contains both $p$ and $q$. To see this, wlog. we can assume that $p$ and $q$ are in different faces of the arrangement. Let $P$ be a point set that contains $p$, $q$ and contains exactly one point in every face of the arrangement $\AA$. Let $\HH$ be the hypergraph whose vertex set is $P$ and there is a hyperedge for each pseudoline containing the points on one of its sides, this is a pseudohalfplane hypergraph (\cite{abafree}, Proposition A.1). Duplicate $p$ and $q$ to get the hypergraph $\HH'$ as in Lemma \ref{lem:discretelevi} on ordered vertex set $P\cup\{p',q'\}$ such that $p'$ (resp. $q'$) is immediately after $p$ (resp. $q$) in the vertex order and a hyperedge contains $p'$ (resp. $q'$) if and only if it contains $p$ (resp. $q$). 

Now we can apply Lemma \ref{lem:discretelevi} on $\HH$ to get the set $X$ which we add as a hyperedge to $\HH'$ to get $\HH''$, a pseudohalfplane hypergraph. There exists a realization of $\HH''$ by pseudohalfplanes (\cite{abafree}, Proposition A.1). In this realization, by Theorem \ref{thm:unique}, the boundaries of the pseudohalfplanes must form an arrangement equivalent to $\AA$ plus an additional pseudoline $l_x$ corresponding to $X$. Observe that $p'$ (resp. $q'$) must be in the same face as $p$ (resp. $q$) in $\AA$. However, the pseudohalfplane corresponding to $X$ contains exactly one of each pair, thus $l_x$ must cross the faces of $\AA$ that contain $p$ and $q$. Now we can redraw locally $l_x$ inside these faces such that $l_x$ goes through $p$ and $q$, as required, while it intersects the same edges of the arrangement the same way as before.

\bigskip

Now we proceed by proving the geometric theorems.

\begin{proof}[Proof of Theorem \ref{thm:weakhellypsconvexplane}]
	Let $P$ be a point set that contains exactly one point in every face of the arrangement of the boundary pseudolines of the pseudohalfplanes of $\EE$. Let $\HH$ be the hypergraph whose vertex set is $P$ and whose hyperedges are the sets $F\cap P$ for $F\in \EE$. $\HH$ is a pseudohalfplane hypergraph (\cite{abafree}, Proposition A.1). Moreover by definition, the sets $C\cap P$ for $C\in \CC$, are convex sets of $\HH$. Thus, Theorem \ref{thm:weakhellypsconvex} implies that $\HH$ can be extended by a vertex $v$ to get the pseudohalfplane hypergraph $\HH'$ such that $v$ is in the intersection of all the extensions of $C\cap P$ for $C\in \CC$.
	Every pseudohalfplane hypergraph has a realization with pseudohalfplanes (\cite{abafree}, Proposition A.1), thus $\HH'$ can be realized as well. It is easy to see that we can even realize it such that whenever two hyperedges are each other's complements then the boundaries of the two corresponding pseudohalfplanes coincide.
	If we forget the image of $v$ from this realization, we get a realization of $\HH$. Theorem \ref{thm:unique} implies that this realization of $\HH$ is equivalent to $\EE$. As in this realization of $\HH'$ the image of $v$ must be in every member of $\CC$, in the realization of $\HH$ this point also must be in every member of $\CC$, and so it is a point as required.
\end{proof}

\begin{proof}[Proof of Theorem \ref{thm:pseudocaratheodoryplane}]
	Let $P$ be a minimal size superset of $P'\cup\{p\}$ such that $P$ contains at least one point in every face of the arrangement of the boundary pseudolines of the pseudohalfplanes of $\EE$. Let $\HH$ be the hypergraph whose vertex set is $P$ and whose hyperedges are the sets $F\cap P$ for $F\in \EE$. $\HH$ is a pseudohalfplane hypergraph (\cite{abafree}, Proposition A.1). Moreover by definition, the sets $C\cap P$ for $C\in \CC$, are the convex sets of $\HH$.
	
	We need to show that $p$ is strongly inside the convex hull of $P$ with respect to $\HH$. Assume on the contrary, then we can extend $\HH$ with a hyperedge $H$ such that with respect to $\HH'=\HH\cup\{H\}$ $p\notin Conv(P)$. We can realize $\HH'$ with points and pseudohalfplanes, then if we forget the image of $H$ from this realization, we get a realization of $\HH$. Theorem \ref{thm:unique} implies that this realization of $\HH$ is equivalent to $\EE$.
	Thus in this realization the image of $H$ gives a pseudohalfplane $E$ such that $p\notin Conv(P)$ when extending $\EE$ with $E$, a contradiction.
	
	Thus Theorem \ref{thm:pseudocaratheodory} implies that there exists a $P''\subseteq P'$, $|P''|\le 3$ such that $p\in Conv(P'')$ with respect to $\HH$. This implies that $p\in Conv(P'')$ also with respect to $\EE$, as claimed.
	
	Moreover, Theorem \ref{thm:pseudocaratheodory} implies that the vertices of $P''$ are extremal in $P'$ with respect to $\HH$. This implies, using an indirect argument like before, that the points of $P''$ are extremal in $P$ with respect to $\EE$.
\end{proof}

\begin{proof}[Proof of Theorem \ref{thm:radonplane}]
	Again, let $P$ be a minimal size superset of $P'\cup\{p\}$ such that $P$ contains at least one point in every face of the arrangement of the boundary pseudolines of the pseudohalfplanes of $\EE$. Let $\HH$ be the hypergraph whose vertex set is $P$ and whose hyperedges are the sets $F\cap P$ for $F\in \EE$. $\HH$ is a pseudohalfplane hypergraph (\cite{abafree}, Proposition A.1). Moreover by definition, the sets $C\cap P$ for $C\in \CC$, are the convex sets of $\HH$.
	
	By Theorem \ref{thm:radon} we can extend $\HH$ with a new vertex $v$ so that it is still a pseudohalfplane hypergraph $\HH'$ and there is a partition of $P'$ into two subsets such that the convex hulls of these subsets (with respect to $\HH$) both contain $v$. We can realize $\HH'$ with points and pseudohalfplanes, then if we forget the image of $v$ from this realization, we get a realization of $\HH$. Theorem \ref{thm:unique} implies that this realization of $\HH$ is equivalent to $\EE$. Thus in this realization the image of $v$ gives a point $p$ such that in the same partition of $P'$ into two subsets $p$ is in the convex hulls of these subsets (with respect to $\EE$), as required.
\end{proof}	

\subsection{Proofs of optimality of Helly-theorems}\label{sec:constr}

We now show that just like the original Helly's theorem in the plane, Theorem \ref{thm:weakhellypsconvex} and Lemma \ref{lem:weakhellypshp} is optimal in the sense that we cannot replace $3$ with $2$ and still hope for one new vertex to hit all the convex sets:

\begin{claim}\label{claim:no2-1}
	There exists a pseudohalfplane hypergraph $\HH$ on vertex set $S$ such that its  hyperedges are pairwise intersecting yet in any extension $\HH'$ of $\HH$ with a new vertex $w$ there is a hyperedge $H\in \HH$ for which $H\cup\{w\}$ is not a convex set of $\HH'$ (in which case it is also not a hyperedge of $\HH'$).
\end{claim}
\begin{proof}	
	Let the vertex set of $\HH$ be $ a,b,c,v$. Let the hyperedges of $\HH$ be the subsets of size $2$ of $\{a,b,c\}$ with and without $v$. This hypergraph on $4$ vertices and $6$ hyperedges is a pseudohalfplane hypergraph, as we can realize it with points and halfplanes in the plane. Indeed, let the points $a,b,c$ be the vertices of an equilateral triangle and $v$ its center. Moreover, every pair of hyperedges has a common vertex. Let $\HH'$ be the hypergraph that we get from $\HH$ by adding $w$ to the vertex set and adding $w$ to some of the hyperedges. Assume on the contrary that for every hyperedge $H\in \HH$ we have that $H\cup\{w\}$ is a convex set of $\HH'$, that is, the intersection of some of the hyperedges of $\HH'$. It is easy to see that the only way for this is that in fact all these sets are hyperedges of $\HH'$, that is, we had to add $w$ to all of the hyperedges of $\HH$ to get $\HH'$.
	We claim that $\HH'$ is not a pseudohalfplane hypergraph. This can be checked by a tedious case analysis where one needs to check for every possible order of the $5$ vertices and for every possible subfamily of the $6$ hyperedges if complementing these hyperedges gives us an ABA-free hypergraph with the given order. We have done this with the help of a computer program.
	
	In case of Lemma \ref{lem:weakhellypshp} the fact that $\HH'$ is not a pseudohalfplane hypergraph immediately shows what we wanted while for Theorem \ref{thm:weakhellypsconvex} we set $\CC=\HH$ and then clearly the only way to get any set $C$ in $\CC$ as an intersection of a subfamily of hyperedges of $\HH$ is if this subfamily contains $C$ and thus to extend every convex set $C$ with $w$ we actually need to extend every hyperedge with $w$.
\end{proof}

Note that an ABA-free hypergraph can be always extended with a vertex that is in every hyperedge (and thus in every convex set of the hypergraph), thus for ABA-free hypergraphs the conclusion of Helly's theorem holds even without any additional assumptions. Similarly, for pseudohalfplanes two additional vertices can always hit every hyperedge without any additional assumptions. Indeed, for any pseudohalfplane hypergraph we can add a vertex which is in exactly the topsets and another which is in exactly the bottomsets, and it is easy to see that it remains a pseudohalfplane hypergraph while these two additional vertices together hit all hyperedges. Thus Lemma \ref{lem:weakhellypshp} is the best meaningful statement in this sense.

On the other hand for convex sets in the plane there is a well-known simple counterexample even for $2\rightarrow +k$: let the $n$ convex sets be $n$ lines in the plane in general position, these are pairwise intersecting yet cannot be hit by less than $n/2$ points. We can convert this construction to the pseudoconvex setting to give also there a counterexample for $2\rightarrow +k$:

\begin{claim}\label{claim:no2-k}
	For every $k$ there exists a pseudohalfplane hypergraph $\HH$ on vertex set $S$ and a subfamily of its convex sets $\CC=\{C_1,\dots, C_l\}$ such that $C_i=\cap \HH_i$ for a subfamily $\HH_i$ of the hyperedges, such that there exists no extension $\HH'$ of $\HH$ onto vertex set $S\cup W$ with $|W|=k$ (for each hyperedge $H$ of $\HH$ and every $w\in W$ we either add $w$ to $H$ or not) such that for all $i$ we have $\cap \HH_i'\cap W\ne \emptyset$ (where $\HH_i'$ contains the extensions of the hyperedges in $\HH_i$, note that $\cap \HH_i'\cap S=C_i$).
\end{claim}
\begin{proof}
	We shall mimic the geometric example. For that take $m=2k+1$ lines in general position in the plane and let $S$ be the set of their intersection points plus one point in each face of the arrangement defined by these lines. Take the (pseudo)halfplane hypergraph defined by the halfplanes on this point set. Thus we get a pseudohalfplane hypergraph $\HH$ on vertex set $S$. Let $C_i$ be the set of vertices on the $i$th line for $1\le i\le m$ and let $\HH_i$ be the subfamily of all hyperedges that contain $C_i$. We claim that these have the required properties. 
	
	Assume on the contrary that we can extend $\HH$ to $W$, a set of $k$ vertices, such that these $k$ vertices hit every $C_i'=\cap \HH'_i$. Then they also need to hit every hyperedge in $\HH_i'$. Now by the pigeonhole principle there are $3$ different convex sets of $\CC$ that are hit by the same added vertex $w\in W$. Now take the intersection points of the three corresponding lines plus one point inside the triangle determined by them. It is easy to see that the hypergraph induced by the four vertices corresponding to these points contains a copy $\HH_0$ of the hypergraph from the proof of Claim \ref{claim:no2-1}. Moreover, $w$ must hit every hyperedge corresponding to the $6$ hyperedges of this copy of $\HH_0$. However, Claim \ref{claim:no2-k} shows that this is not possible, finishing the proof.	
\end{proof}

Notice that this is a slightly weaker type of counterexample then the one in Claim \ref{claim:no2-1} as there we did not require that the extended convex sets are the intersections of the extensions of the same subfamilies of (extended) hyperedges.

\smallskip
Finally, about pseudohemisphere hypergraphs, we show that Theorem \ref{thm:weakhellypsconvex} and Lemma \ref{lem:weakhellypshp} is optimal in the sense that we cannot replace $4$ with $3$ and still hope for one new vertex to hit all the convex sets:

\begin{claim}
	There exists a pseudohemisphere hypergraph $\HH$ on vertex set $S$ such that every triple of its hyperedges intersects yet in any pseudohemisphere extension $\HH'$ of $\HH$ with a new vertex $w$ there is a hyperedge $H\in \HH$ for which $H\cup\{w\}$ is not a convex set of $\HH'$ (in which case it is also not a hyperedge of $\HH'$).
\end{claim}

\begin{proof}
	Let $\HH^*$ be the hypergraph on four vertices that has all subsets as hyperedges except the empty set and the set containing all four vertices. This hypergraph with $14$ hyperedges on $4$ vertices is a pseudohemisphere hypergraph. We can verify this by definition: take $X$ to be the set of the 2nd and 4th vertex in the ordering and $\FF$ be the ABA-free hypergraph containing the size-$2$ hyperedges that contain the first vertex and all the size-$3$ hyperedges; alternatively, we can realize it even by points and hemispheres: put the points on the sphere as vertices of a tetrahedron and then drawing the hemispheres is quite straightforward. Now take the dual $\HH$ of $\HH^*$ (containments are reversed), which has $4$ hyperedges, it is also a pseudohemisphere hypergraph (\cite{abafree}, Proposition 4.4). Every triple of these $4$ hyperedges intersects as this is equivalent to the fact that in $\HH^*$ every triple of points is contained in some hyperedge. If in an extension $\HH'$ of $\HH$ with a vertex $w$ for a hyperedge $H$ of $\HH$ we have that $H\cup\{w\}$ is a convex set of $\HH'$, then this can happen only if $H\cup\{w\}$ is a hyperedge of $\HH'$. Thus, we are left to prove that we cannot extend $\HH$ with a vertex $w$ contained in all $4$ hyperedges. To see this, suppose such an extension exists, then taking again the dual we get a ${\HH^*}'$ with $15$ hyperedges, which has the same hyperedges as $\HH^*$ plus an additional hyperedge containing all four vertices. This has a realization with points and pseudohemispheres in which the boundary pseudocircles define at least $15$ cells, contradicting the easy to see fact that four pseudocircles define at most $14$ cells on the sphere.
\end{proof}

\section{Relations to other extensions of convexity}\label{sec:relations}

\textbf{Topological affine planes (TAPs).}
Let us first summarize the framework of this paper. Given a finite family of pseudohalfplanes and a set of points in the plane, by \cite{abafree} we can define the corresponding pseudohalfplane hypergraph. In this setting we define convex subsets of the vertices and show various combinatorial results which mimic classical results of convexity in the plane. These in turn imply the respective result about convex sets of the pseudohalfplane family. 

However, it is possible to go in the opposite direction (i.e., more geometry instead of less) to arrive to the same conclusion. A topological affine plane (TAP) is basically a continuous extension of a pseudoline arrangement, where every pair of points lies on a unique pseudoline. We provide here the exact definition of a topological affine plane for the convenience of the reader, based on \cite{TAPconvex}. The pseudolines of a TAP on the plane are the interiors of an infinite collection of simple curves (Jordan arcs) on the extended projective plane (we can regard this as a closed disk whose boundary forms the points `at infinity' and whose interior is homeomorphic to the plane) with the following properties. The endpoints of each curve are antipodal points at infinity (and so each pseudoline cuts the plane into two components). Any two curves intersect once, and in the intersection point cross or they share their endpoints (at infinity) and are disjoint otherwise. We require first that for every pair of points $\{x, y\}$ of the plane there exists a unique pseudoline containing $x$ and $y$ and that this pseudoline depends continuously on $x$ and $y$ in the Hausdorff metric. We require also that the intersection of two pseudolines depends continuously on the two pseudolines. 

Note that any finite subset of the pseudolines of a TAP forms a finite pseudoline arrangement as defined in Section \ref{sec:geomconseq}. Conversely, in \cite{TAP} it is shown that every finite pseudoline arrangement can be extended to a TAP. In \cite{TAPconvex} they investigated extensions of classical convexity results to TAPs. Given two points, the part of the unique pseudoline connecting the two points is called the pseudosegment defined by the two points. Using this, in a TAP we can define convexity as we usually do in the plane: a set of points is convex if for every pair of its points the set also contains the pseudosegment connecting these two points. Recall that given a finite pseudohalfplane family, we defined its convex sets as the intersections of subfamilies of the pseudohalfplanes. It is not hard to see that these sets are also convex in any extension to a TAP with the above definition that uses pseudosegments. 

Given a pseudohalfplane hypergraph, we can take any realization with points and pseudohalfplanes (by \cite{abafree} this exists) and then extend this to a TAP.
When doing this, the family of convex sets may change. More precisely, given a pseudohalfplane hypergraph $\HH$, it can be represented by a point set $P$ and a family of pseudohalfplanes. This family of pseudohalfplanes has an extension to a TAP. The pseudohalfplanes of this TAP induce on $P$ a pseudohalfplane hypergraph $\HH'$ whose subhypergraph is $\HH$, but we cannot guarantee that it is equal to $\HH$. I.e., when extending to a TAP, new hyperedges might appear. When adding hyperedges, we may get a bigger family of convex sets and the convex hull of a set of points may get smaller. Thus, given a pseudohalfplane hypergraph, while we are able to translate it to a pseudohalfplane family and then to a TAP, during this process we may get a different statement in this translated setting. The advantage of the original abstract setting is that it can directly handle pseudohalfplane hypergraphs that are subhypergraphs of pseudohalfplane hypergraphs that arise from TAPs.

In \cite{TAPconvex} they prove Helly's Theorem, Carath\'eodory's Theorem, Kirchberger's Theorem, Separation Theorem, Radon's Theorem for TAPs. We have seen that they may not imply directly our results. However, it is quite straightforward to check how the conditions and conclusions of our theorems change when adding new hyperedges to the pseudohalfplane hypergraph, and see that every one of these results on a $\HH'$ pseudohalfplane hypergraph implies the same on its subhypergraph $\HH$. Thus, earlier results about TAPs do imply our respective results about pseudohalfplane families and in turn also about pseudohalfplane hypergraphs. 
Furthermore, in \cite{pseudoerdosszekeres} they consider the Erd\H os-Szekeres problem in TAPs and prove a result which is best known even in the plane. In addition, a  mostly combinatorial proof using wiring diagrams of the Cup-Cap Theorem for finite families of pseudohalfplanes (i.e. the geometric equivalent of Theorem \ref{thm:pseudoesz}) can be found in \cite{dobbins}.

In the reverse direction, e.g., by approximating TAPs by finite families of pseudolines and applying Theorem \ref{thm:weakhellypsconvexplane} for each of them we get a set of points which has at least one accumulation point and this has the required properties of Helly's Theorem for TAPs. On the other hand, using such simple approximation arguments does not seem to be enough to show all the other aforementioned results for TAPs using our discrete results.

Overall, it turns out that most of our results (in particular Helly's Theorem, Carath\'eodory's Theorem, Kirchberger's Theorem, Separation Theorem, Radon's Theorem and the Cup-Cap Theorem) are proved already in the context of TAPs. However, the proofs about TAPs rely on the continuous geometry of TAPs whereas our results are strictly about discrete combinatorial structures. Indeed, while many times they both resemble the original planar proofs, they also feel very different in nature, with a geometric reasoning on one side and a combinatorial reasoning on the other. We think that this completely different approach can lead to new insights and thus to new results, possibly improving even old results about convex sets in the plane.

\smallskip
\textbf{Oriented matroids of rank $3$.}
Our treatment of oriented matroids is based on the book by Bj\"orner et al. \cite{bjorner}, which we also recommend for the interested reader. Chapter 6 of \cite{handbook} by Richter-Gebert and Ziegler gives a more brief introduction to the area. We try to be as concise as possible and so we omit many of the definitions and technicalities, in particular we concentrate on oriented matroids of rank $3$. Nevertheless, before we can phrase the connection, we need to give some introduction to oriented matroid theory.

Oriented matroids were invented multiple times independently and have several different but equivalent axiomatizations. One way to define an oriented matroid $\MM$ is on a vertex set $S$ by its set of \emph{circuits}, which is a set of sign-vectors, each of which assigns one of $\{-,+,0\}$ to each vertex. These sign-vectors must satisfy the so-called circuit axioms. A certain composition rule of circuits generates the \emph{vectors} of $\MM$, a set of sign-vectors that includes all the circuits. By a certain \emph{orthogonality rule} one can define the covectors as all the sign-vectors that are orthogonal to every vector of $\MM$. A special subset of covectors is the set of \emph{cocircuits}. An oriented matroid can be equivalently defined via its circuits, vectors, cocircuits or covectors (among others). The support of a sign-vector is the set of vertices where it is non-zero. A maximal covector or \emph{tope} of a matroid $\MM$ is a covector whose support is maximal with respect to inclusion.

One can define the rank of an oriented matroid. We omit the definition, instead we give the geometric intuition of the rank, before which we need to introduce some notions. Following \cite{bjorner}, Section 1.4, a subset $S$ of $S_d$ is called a \emph{pseudosphere} if $S = h(S_{d-1})$ for some homeomorphism $h: S_d \rightarrow S_d$, where $S_{d-1} = \{x \in S_d : x_{d+1} = 0\}$. $S_d\setminus S$ has two connected components, $S^+$, $S^-$, the sides of $S$, which we call (open) \emph{pseudohemispheres}. A finite family of pseudospheres $\cal S$ is called an \emph{arrangement of pseudospheres} if every non-empty intersection $S'$ of a subfamily of the pseudospheres is (homeomorphic to) a sphere of some dimension and for every $S\in \cal S$ with $S'\not \subseteq S$, $S'\cap S$ is a pseudosphere in $S'$ with sides $S'\cap S^+$ and $S'\cap S^-$. The pseudohemispheres of an arrangement of pseudospheres is called an \emph{arrangement of pseudohemispheres}.
In dimension $2$ (which case corresponds to oriented matroids of rank $3$) this can be regarded as a family of regions on $S^2$ whose boundaries are centrally symmetric simple curves such that any two intersect exactly twice (see also the remark after Definition \ref{def:pshemi}). Note that in \cite{abafree} for a pseudohemisphere hypergraph a very similar representation is given: vertices are mapped to points and hyperedges are mapped to pseudohemispheres of an arrangement of pseudohemispheres on $S^2$ such that containments are preserved (and this mapping can again be reversed). Assuming that in each cell there is exactly one point, this gives a mapping between an oriented matroid and the pseudohemisphere hypergraph.

The Topological Representation Theorem by Folkman and Lawrence says that one can represent a rank $d+1$ simple\footnote{The definition of a simple oriented matroid we also omit.} oriented matroid such that the vertex set is mapped to an arrangement of open pseudohemispheres on the $d$-dimensional sphere $S^d$ and each tope is mapped to a cell of the arrangement defined by these pseudohemispheres such that it preserves signatures (i.e., given the open pseudohemisphere $H_v$ corresponding to vertex $v$ and a cell $K_c$ corresponding to the tope $c$, the sign of $K_c$ on $v$ is `$+$' if $H_v$ contains $K_c$ and `$-$' otherwise) and vice versa (i.e., a family of pseudohemispheres on $S^d$ can be mapped to an oriented matroid of rank $d+1$).

On the set of topes of a matroid there is a so-called $T$-convexity defined. Recall that a tope corresponds to a cell in the above representation. Moreover, it turns out that on the set of topes, a subset is defined to be $T$-convex exactly if in the above representation it is the intersection of pseudohemispheres, thus we can phrase, e.g., Theorem \ref{thm:weakhellypsconvexplane} in terms of topes and $T$-convex sets of oriented matroids. However, we did not find (at least not in \cite{bjorner}) any appearance of classical convexity results about $T$-convex sets.

The above connection between oriented matroids of rank $3$ and pseudohemisphere hypergraphs via the geometric representation by pseudohemispheres on $S^2$ is dual in the following sense: vertices of the matroid correspond to pseudohemispheres which in turn correspond to hyperedges of a pseudohemisphere hypergraph. However, by \cite{abafree} the dual (for containment)\footnote{Here one has to be cautious with this notion of dual where we map vertices into hyperedges and vice versa, as the notion of dual in matroids has a very different meaning. In fact the notion of an adjoint of a matroid is the one similar to hypergraph duality.} of a pseudohemisphere hypergraph is also a pseudohemisphere hypergraph, so there is hope for a more straightforward connection (where vertices of the matroid are mapped to vertices of the hypergraph). And indeed, in addition to the above well-known geometric representation of matroids, which is called in \cite{bjorner} a Type I representation (which every rank $d+1$ matroid has), there is also a less known dual one using the so-called pseudoconfiguration of points, which is called a Type II representation (\cite{bjorner}, Definition 5.3.1.). 

\begin{defi}
	A \emph{pseudoconfiguration of points} of rank $d+1$ is a pair $(\AA,P)$ where $\AA$ is an essential (i.e., they don't have a common intersection) arrangement of pseudospheres (i.e., the boundaries of a family of pseudohemispheres) on $S^d$, and $P$ is a set of at least $d$ points on $S^d$ such that: every $d$-tuple of points from $P$ is on some pseudosphere of $\AA$ and every pseudosphere of $\AA$ contains a subset of $P$ not contained in any other pseudosphere of $\AA$.
\end{defi}

Every pseudoconfiguration of points uniquely determines an oriented matroid on $P$, whose cocircuits are the following: for each pseudosphere and choice of positive side of this pseudosphere, there is a sign-vector which is `$+$' for a point if it is on the positive side, `$-$' when it is on the negative side and `$0$' when it is on the pseudosphere. A Type II representation of an oriented matroid $\MM$ of rank $d+1$ is a pseudoconfiguration of points that determines $\MM$ this way.

Unlike Type I representation, not every oriented matroid of rank $d+1$ has a Type II representation: $\MM$ has one if and only if $\MM$ satisfies the intersection property IP$_0$, which is equivalent to $\MM$ having an adjoint (we omit the definition, the proof uses the existence of the Type I representation for its adjoint, see \cite{bjorner}). Luckily a rank $3$ oriented matroid always satisfies IP$_0$ (\cite{bjorner} Proposition 6.3.6) and thus has a Type II representation.

Given a pseudohemisphere hypergraph, it can be represented by a set of points and a family of pseudohemispheres on $S^2$ by \cite{abafree}. The boundaries of the pseudohemispheres form an arrangement of pseudospheres (also called pseudogreatcircles when on $S^2$). By extending this with additional points and pseudogreatcircles, we can get a pseudoconfiguration of points of rank $3$. This uniquely determines an oriented matroid $\MM$ of rank $3$.

There is a convex-set definition for oriented matroids. By \cite{bjorner}, Theorem 9.2.1, which collects results by Las Vergnas and Bachem and Wanka, the convex sets in a Type II representation are exactly those point sets that are intersections of closed pseudohemispheres determined by the pseudogreatcircles of the pseudoconfiguration of points. This theorem also implies that if the rank $3$ oriented matroid is \emph{acyclic}, then for these convex sets Carath\'eodory's Theorem, Radon's Theorem, Helly's Theorem, Separation Theorem all hold.\footnote{In fact the theorem proves this for acyclic oriented matroids with the so-called Generalized Euclidean intersection property IP$_2$ (\cite{bjorner}, Definition 7.5.2) of arbitrary fixed rank $r$, not just $r=3$. As mentioned before, every acyclic oriented matroid of rank $3$ has the IP$_0$ property and IP$_0$ implies IP$_2$ (\cite{bjorner}, Proposition 7.5.3).} The definition of being acyclic when translated to the corresponding pseudoconfiguration of points means that there is a pseudohemisphere that contains all points.

Given a pseudohalfplane hypergraph, we can add the hyperedge containing all vertices and then by possibly adding further hyperedges and vertices we get a pseudohalfplane hypergraph (which is also a pseudohemisphere hypergraph) that can be mapped to an oriented matroid of rank $3$ which is also acyclic as there is a pseudohemisphere containing all points.

Then the above convexity results about acyclic oriented matroids imply our convexity results for this extended pseudohalfplane hypergraph. Finally, similar to how the results about TAPs implied our results, these also imply the respective results about the original pseudohalfplane hypergraph.

Summarizing, when trying to show our convexity results about pseudohalfplane hypergraphs via oriented matroids, we first had to extend our hypergraph so that it becomes an (acyclic rank $3$) oriented matroid. Only then could we apply the results about oriented matroids. Also for the same reason, i.e., the structure becoming richer, the conclusion is not exactly what we need, but manually checking the statements it turns out that they imply those in a straightforward way.

What about the reverse direction? Given an acyclic oriented matroid of rank $3$ it has a Type II representation, by adding perturbed copies of the pseudogreatcircles we can easily achieve that the point sets defined by closed pseudohemispheres of the original pseudoconfiguration are defined by open pseudohemispheres of the new configuration. Then we can regard this as an arrangement of points and pseudohalfplanes and then our convexity results imply the respective results about the original oriented matroid of rank $3$.

Thus for rank $3$ acyclic oriented matroids Theorem 9.2.1 of \cite{bjorner} is formally weaker (as it talks only about certain pseudohalfplane hypergraphs) but practically equivalent to our set of results.

For another comparison of our approach and oriented matroids, we note that oriented matroid axioms are global in the sense that, e.g., for a pair of signed sets they guarantee a third one with a certain property. On the other hand, once the signature of the hyperedges and vertices is fixed (i.e. the family of topsets within the hyperedges and the set $X$ within the vertex set), our definitions are local, that is they only say that every pair of hyperedges must have a certain property. This is one particular advantage of our approach, which comes from the fact that pseudohemisphere and pseudohalfplane hypergraphs in general can be proper subfamilies of such hypergraphs that come from an oriented matroid.

Without going into more details, we mention that extremal vertices were also defined for acyclic oriented matroids (see \cite{bjorner} for further details). 

Recently, generalizations of oriented matroids called complexes of oriented matroids were defined \cite{comdef}. Kirchberger's Theorem was proved in this more general context \cite{com-kirchberger}.

\smallskip
\textbf{P-convex hulls of pseudolines.}
There is yet another way in the literature to define convexity and state respective results on pseudolines arrangements. This approach circumvents the notion of a pseudohalfplane. To do this, it needs to work in the dual setting. Given a finite arrangement of pseudolines $\AA$ and a point $p$ not contained in any member of $\AA$, we say that a pseudoline $l\in \AA$ is in the $p$-convex hull of a subfamily $\BB \subset \AA$ if every path from $p$ to a point of $l$ meets some member of $\BB$. About $p$-convex hulls many classical results are proved, e.g., Helly's Theorem, Carath\'eodory's Theorem, Kirchberger's Theorem, Separation Theorem, Radon's Theorem among others (see Chapter 5 of \cite{handbook} by Felsner and Goodman for references).

One immediately sees that the $p$-convex hull is a dual notion as instead of defining convex hulls of points we define convex hulls of pseudolines (similar to the Type I representation of oriented matroids). We can make this more formal. Define the pseudohalfplane family $\EE$ such that for each pseudoline of $\AA$ we take its side which contains $p$. Put a point into every face of the arrangement $\AA$ to get the point set $P$ (we also assume that $p\in P$). Now $\EE$ defines a pseudohalfplane hypergraph $\HH$ on $P$. We get a pseudohalfplane hypergraph whose every hyperedge contain $p$ (more precisely the vertex corresponding to $p$). It was shown in \cite{abafree} that the dual of a pseudohalfplane hypergraph $\HH$ is also a pseudohalfplane hypergraph $\hat \HH$ provided there is a point in $P$ contained by all hyperedges of $\HH$\footnote{In fact in \cite{abafree} this is proved if all hyperedges avoid a point $p$ but taking the complement of every hyperedge of $\HH$ we can apply this to get the same in case all hyperedges contain a point $p$.}. Now it is easy to see that a line $l\in \AA$ is in the $p$-convex hull of $\BB$ if and only if the hyperedge $h_l$ of $\HH$ corresponding to $l$ contains the hyperedge corresponding to $b$ for every $b\in \BB$, that is if a vertex $q\in P$ is in the hyperedge corresponding to $b$ for every $b\in \BB$ then $q$ is also in $h_l$. In $\hat \HH$ this is equivalent to saying that if a  hyperedge $h_q$ contains the vertex $\hat b$ corresponding to $b$ for every $b\in \BB$ then it also contains $\hat h_l$, the vertex corresponding to $l$. Notice that this is equivalent to saying that $\hat h_l\in Conv(\hat \BB)$ where $\hat \BB$ is the family of vertices of $\hat \HH$ that are the duals of the hyperedges of $\HH$ corresponding to the pseudolines in $\BB$.

Thus, by identifying the pseudolines with the pseudohalfplanes that have this pseudoline as a boundary and contain $p$ and then taking the dual, we got back to our notion of a convex hull. Without going into details we claim that from this it easily follows that, e.g., our Carath\'eodory's theorem implies the respective Carath\'eodory's theorem about $p$-convex hulls. On the other hand for results which introduce a new object, e.g., Helly's Theorem, this does not work. While we can do the dualization and apply our result in this dual setting to find the required vertex, afterwards it is not clear how we can dualize back this to get a required pseudoline in the original setting (again there is the limitation to when we can dualize). Most probably in such a situation we can instead repeat the abstract proof directly in the primal setting, dualizing the arguments. However, we did not work out the details. Overall, some results about $p$-convex hulls follow from our results just by dualization while for others it is not clear if and how this can be done, in which cases one may be able to dualize the arguments to give abstract proofs for such results about $p$-convex hulls.

One can also ask about implications in the other direction, that is, if a result about $p$-convex hulls implies our respective result. The problem with this is again that in order the dual of a pseudohalfplane hypergraph to also be a pseudohalfplane hypergraph, we needed the special point that is in every hyperedge. Thus, if such a point does not exist we cannot dualize and so we do not know how a result about $p$-convex hulls would imply a result about our setting. Thus in this sense, e.g.. Carath\'eodory's Theorem about $p$-convex hulls is a special case of our Carath\'eodory's Theorem. Recall that $p$-convex hulls can be phrased using pseudohalfplanes which all contain a common point, which resembles more the easier special case of upwards pseudohalfplanes (which correspond to ABA-free hypergraphs) than the general case of pseudohalfplanes.

\smallskip
\textbf{ABAB-free hypergraphs.}
Our theory is built on the notion of ABA-free hypergraphs. Similar to them, in \cite{abafree} ABAB-free (and ABABA-free etc.) hypergraphs were defined and in \cite{abab} it was shown that they are equivalent to hypergraphs defined on a point set by pseudodisks all containing the origin. Do some of our results about halfplanes and convex sets translate to ABAB-free hypergraphs? We note that one important property of pseudohalfplane hypergraphs used in \cite{abafree} is that they have shallow hitting sets (for definitions see \cite{abafree}). On the other hand, it was also proved there that ABAB-free hypergraphs do not admit shallow hitting sets. Nevertheless, some positive results were proved in \cite{abab}.

\smallskip
\textbf{Convexity spaces.}
There are various other abstract relaxations of convexity. These usually cover a significantly wider range of settings, while still having Helly-type etc. properties, albeit this comes at a price. Most relevant to us is the notion of a convexity space (also called as a closure structure). A {\it finite convexity space} $\CC$ is a hypergraph on a finite vertex set $S$ such that the $\emptyset,S\in \CC$ and for every pair of hyperedges $A,B\in \CC$ $A\cap B\in \CC$. Here the convex hull of a subset of vertices is defined as the intersection of  the convex sets (i.e., hyperedges of $\CC$) containing it. Notice that for any hypergraph $\HH$ on vertex set $S$ containing $\emptyset $ and $S$, the convex sets of $\HH$ defines according to Definition \ref{def:pseudoconvex} form a convexity space, where the convex hulls of this convexity space are the same as what we get from Definition \ref{def:pseudoconvex}. In this respect we regard a special convexity space, which is defined as the convex sets of pseudohalfplane hypergraphs and prove analogs of classical planar convexity results. While convexity spaces in general have a lot of interesting properties, it is a much more general framework and thus allows only to show weaker or only conditional variants of classical planar convexity results. It is beyond the scope of this paper to go into further details, we refer the reader to the book \cite{convexstructures} for an overview of the theory of convexity spaces.

\smallskip
\textbf{Further extensions of convexity.}
As classical convex geometric theorems have a very central role, it is impossible to list every variant and extension. Thus, without the aim of completeness, we list a few more of them that are more recent or seem closer to our results.

The relation of points and pseudolines in the plane can be encoded by allowable sequences of permutations. One can again define convexity, in particular the Erd\H os Szekeres problem was regarded in this setting (for references see Chapter 5 of \cite{handbook} by Felsner and Goodman).

In \cite{barany2021hellytype} they survey recent developments related to classical convexity results. In particular, following \cite{MY}, given a hypergraph (with bounded VC-dimension), whose hyperedges represent `halfspaces', they define MY-convexity where the convex sets are the intersections of these halfspaces. This is exactly the same as our central Definition \ref{def:pseudoconvex}.

In \cite{bergold2021topological} they discuss classical theorems about convexity in the context of topological drawings.

\section{Discussion}\label{sec:discussion}

We have presented a method to generalize statements about discrete point sets, halfplanes and convex sets to statements on discrete points sets, pseudohalfplanes and pseudoconvex sets. Our setting is purely combinatorial and built using elementary parts, starting with the notion of a hypergraph being ABA-free everything else is built up step-by-step. This offers a very simple axiomatization of planar (pseudo)convexity. We managed to generalize this way many classical results about planar convexity. This discrete relaxation of planar geometry is significantly more general than the planar setting yet still allows us to prove statements that are exactly like their geometric counterparts. Also, many natural families of regions are pseudohalfplane families, thus our generalizations have immediate geometric consequences (e.g., about translates of an unbounded convex region). We compared our results to other similar results about TAPs, oriented matroids and $p$-convex hulls, pinpointing the connections and differences between them.

There are many further important results about convex sets, yet it falls beyond the scope of this paper to consider all of them for pseudoconvex sets. We list a few of these possibilities: colorful Carath\'eodory's Theorem, colorful Helly Theorem, Tverberg's Theorem, colorful Tverberg's Theorem, fractional versions, results about the existence of empty $k$-holes; are these true for pseudoconvex sets? It is easy to check that the colorful Helly Theorem for pseudoconvex sets follows from a result of in (\cite{Bokowski2011}, Theorem 5). The rest of these problems we leave as open problems. Further, it is interesting to see if problems open about convex sets can be improved in the context of pseudoconvex sets, in particular finding a maximal subset of points in a convex position, bounding the size of weak epsilon-nets for convex sets or the number of k-sets. We mentioned earlier that the best known result about the size of a maximal convex subset of points does extend to the geometric pseudohalfplane setting via TAPs \cite{pseudoerdosszekeres} and thus also to pseudohalfplanes. It would be interesting to give a direct proof of the respective statement about pseudohalfplane hypergraphs using our methods.

It would be interesting to develop a similar framework of higher dimensional discrete pseudoconvexity. On one hand, it is well-known (see, e.g., the discussion in \cite{TAPconvex}) that in dimension $d\ge 3$ every topological affine $d$-space is isomorphic to the Euclidean $d$-space, on the other hand convexity results do hold for acyclic oriented matroids of any given rank \cite{bjorner}. We do not know if those can be also rephrased in a similar way to how we treat pseudohalfplane hypergraphs, that is, using some kind of ordering of the vertices and some forbidden alternations. Thus an abstract theory similar to ours for higher dimensional pseudohalfspaces and convex sets may be possible to phrase.

Finally, we mention that several computational and enumerational problems arise related to ABA-free and pseudohalfplane hypergraphs. Let us mention a few of them. First, how efficiently can one decide/witness that a hypergraph is ABA-free? Second, given a hypergraph on an ordered (resp. unordered) set of vertices, how efficiently can one decide/witness if this is a pseudohalfplane hypergraph with some choice of what are its top and bottom vertices (resp. also choice of vertex order)? What is the possible number of such valid choices if the hypergraph is known to be a pseudohalfplane hypergaph?

\subsubsection*{Acknowledgement}

The author is grateful to D. P\'alv\"olgyi for the many discussions about these results and to A. Holmsen for the several insightful comments (primarily but not exclusively) about the connections to earlier results, especially to oriented matroids; and finally to the reviewers whose feedback made the presentation much better.

\bibliographystyle{abbrvnat}
\bibliography{psdisk}

\end{document}